\documentclass[reqno]{amsart}

\usepackage{amssymb,amsthm,amstext,latexsym,amsmath}
\usepackage{enumerate} 
\usepackage{enumitem}
\usepackage[mathscr]{eucal}
\usepackage{tensor}
\usepackage[all]{xy}
\usepackage{color,wrapfig} 
\usepackage{float} 
\usepackage[pdftex]{graphicx} 
\usepackage{longtable}
\usepackage{mathtools} 

\input xy
\xyoption{all}

\theoremstyle{plain}
\newtheorem{thm}{Theorem}
\newtheorem{theorem}[thm]{Theorem}
\newtheorem{corollary}[thm]{Corollary}
\newtheorem{lemma}[thm]{Lemma}
\newtheorem{prop}[thm]{Proposition}
\newtheorem{proposition}[thm]{Proposition}
\newtheorem*{mr}{Main Result}
\newtheorem*{introth}{Theorem}
\newtheoremstyle{exm}
{9pt}{9pt}{}{}{\bfseries}{}{.5em}{}
\theoremstyle{exm}
\newtheorem{exm}[thm]{Example}
\newtheoremstyle{rmk}
{9pt}{9pt}{}{}{\bfseries}{}{.5em}{}
\theoremstyle{rmk}

\theoremstyle{alg}

\newtheoremstyle{question}
{9pt}{9pt}{}{}{\bfseries}{}{.5em}{}
\theoremstyle{question}

\numberwithin{equation}{section}
\numberwithin{thm}{section}
\numberwithin{figure}{section}

\theoremstyle{definition}
\newtheorem{definition}[thm]{Definition}
\newtheorem{defin}[thm]{Definition}
\newtheorem{example}[thm]{Example}
\newtheorem{remark}[thm]{Remark}

\newcommand{\C}{\mathbb{C}}

\newcommand{\R}{\mathbb{R}}
\newcommand{\T}{\mathbb{T}} 
\newcommand{\Z}{\mathbb{Z}}

\def\epsilon{\varepsilon}

\newcommand{\al}{\alpha}

\newcommand{\ep}{\epsilon}

\newcommand{\lam}{\lambda}

\newcommand{\om}{\omega}

\newcommand{\msT}{\mathscr T}

\newcommand{\sm}{(M,\omega)} 
\newcommand{\cotan}{\mathrm{T}^*} 
\newcommand{\diff}{\mathrm{d}} 

\newcommand{\Hs}{\mathcal{H}am_{S^1}} 
\newcommand{\cp}{\C\mathrm{P}^2} 
\newcommand{\st}{\mathcal{S}\mathcal{T}} 
\newcommand{\semitoric}{(M,\omega,\Phi)} 
\newcommand{\hamiltonian}{(M,\omega,J)} 
\newcommand{\polygon}{\mathcal{P}_{\boldsymbol{\ep}}} 

\DeclarePairedDelimiter\abs{\lvert}{\rvert}

\newcommand{\Jmin}{J_{\mathrm{min}}}
\newcommand{\Jmax}{J_{\mathrm{max}}}

\newcommand{\purge}[1]{} 

\newcommand{\vungoc}{V\~u Ng\d{o}c}

\title[From semi-toric systems to Hamiltonian $S^1$-spaces]{From semi-toric systems to Hamiltonian $S^1$-spaces}
\author[S. Hohloch]{Sonja Hohloch}
\address{Department of Mathematics, EPFL, Lausanne, Switzerland}
\email{sonja.hohloch@epfl.ch}
\author[S. Sabatini]{Silvia Sabatini}
\address{Department of Mathematics, EPFL, Lausanne, Switzerland}
\email{silvia.sabatini@epfl.ch}
\author[D. Sepe]{Daniele Sepe}
\address{CAMGSD, Instituto Superior T\'ecnico, Lisbon, Portugal}
\email{dsepe@math.ist.utl.pt}

\date{\today}

\begin{document}

\begin{abstract}
This paper studies the local and global aspects of semi-toric integrable systems, introduced by \vungoc, using ideas stemming from the theory of Hamiltonian $S^1$-spaces developed by Karshon. First, we show how any labeled convex polygon associated to a semi-toric system (as defined by \vungoc) determines Karshon's labeled directed graph which classifies the underlying Hamiltonian $S^1$-space up to isomorphism. Then we characterize adaptable semi-toric systems, i.e.\ those whose underlying Hamiltonian $S^1$-action can be extended to an effective Hamiltonian $\mathbb{T}^2$-action, as those which have at least one associated convex polygon which satisfies the Delzant condition.
\end{abstract}

\maketitle

\tableofcontents


\section{Introduction}

\label{sec:introduction}

This paper studies the relation between a certain family of completely integrable Hamiltonian systems on closed 4-dimensional symplectic manifolds and Hamiltonian $S^1$-actions on these spaces. As such, it lies at the intersection of the theory of Hamiltonian torus actions on closed symplectic manifolds and the classification of completely integrable Hamiltonian systems. The former is a special case of Hamiltonian actions in symplectic and Poisson geometry, an area of mathematics which brings together algebraic geometry, Lie theory, Poisson geometry and differential topology amongst others. Of particular prominence for the purposes of this work is Karshon's monograph \cite{karshon} on Hamiltonian circle actions on closed 4-dimensional symplectic manifolds, whose results have been extended to higher dimensions (cf. Karshon $\&$ Tolman \cite{karshon-tolman1,karshon-tolman2,karshon-tolman3}). The classification of completely integrable Hamiltonian systems is a driving question in Hamiltonian mechanics with many different aspects to it, which, for the sake of brevity, are not mentioned here (cf. Bolsinov $\&$ Oshemkov \cite{bo} and Pelayo $\&$ \vungoc\ \cite{pelayo-vu ngoc bulletin} for further details). This article is concerned with topological and symplectic aspects of completely integrable Hamiltonian systems, which have been studied since the work on constant energy surfaces by Fomenko \cite{fomenko} and his school. From the point of view of integrable systems, at the heart of this paper lie both the work on local normal forms near non-degenerate singular points by Eliasson \cite{eliasson thesis, eliasson}, and by Miranda $\&$ Zung \cite{miranda-zung}, and ideas which underpin Pelayo $\&$ \vungoc's recent classification of `generic' semi-toric systems (cf. \cite{pelayo-vu ngoc inventiones, pelayo_vu_ngoc_con,vu_ngoc_ff,vu ngoc}). Moreover, the symplectic perspective of Symington \cite{symington} and of Leung $\&$ Symington \cite{leung_symington} on a larger family of integrable Hamiltonian systems has also influenced the approach in this article. \\

Throughout this introduction, let $\sm$ be a connected, closed, symplectic 4-manifold. A \textbf{Hamiltonian $S^1$-space} consists of a triple $(M,\omega,J)$, where $J : M \to \R$ is the moment map of an effective Hamiltonian $S^1$-action (details in Section \ref{def karshon graph}). Such spaces are classified, up to a suitable notion of isomorphism, by Karshon \cite{karshon}, and their invariants are encoded in certain `labeled directed graphs'. 

On the other hand, a {\bf semi-toric system} consists of a triple 
\begin{equation*}
(M,\omega, \Phi =(J,H)) \quad \mbox{with } \Phi: \sm \to \R^2
\end{equation*}
where $\Phi$ defines a completely integrable Hamiltonian system whose singularities are non-degenerate in a suitable sense, and such that $\hamiltonian$ is a Hamiltonian $S^1$-space (details in Section \ref{sec:compl-integr-hamilt}). These systems are introduced and studied by \vungoc\ \cite{vu ngoc}, whose main motivation came from the example of two coupled angular momenta considered by Sadovski{\'{\i}}  $\&$ Z{\^h}ilinski{\'{\i}} \cite{sad_zhi}. Semi-toric systems have not been classified in full generality; however, Pelayo $\&$ \vungoc\ \cite{pelayo-vu ngoc inventiones,pelayo_vu_ngoc_con} achieve a classification under some `generic' conditions. At any rate, \vungoc\ \cite{vu ngoc} already associates to a semi-toric system $\semitoric$ a family of labeled convex polygons, i.e.\ convex polygons together with some marked interior points. This is in analogy with Delzant's \cite{delzant}  classification of {\bf symplectic toric manifolds}, i.e.\ triples $(M,\omega, \mu)$, where $\mu = (\mu_1,\mu_2): \sm \to \R^2$ is the moment map of an effective Hamiltonian $\mathbb{T}^2$-action (cf.\ Definition \ref{defin:stm}). The invariants of a triple $(M,\omega,\mu)$ are encoded in a so-called Delzant polygon which is the image of the moment map (cf. Definition \ref{defin:delzant_polytopes}). The labeled convex polygons of semi-toric systems generalize Delzant's polygons associated to symplectic toric manifolds. However, there are two significant differences, both due to the richer behaviour of semi-toric systems caused by the presence of focus-focus singular points (cf. Section \ref{sec:str}, or \vungoc\ \cite{vu_ngoc_ff} and Zung \cite{zung focus focus} for a definition). On the one hand, there may be several labeled convex polygons associated to a semi-toric system, as there are choices involved (cf. the discussion leading to Theorem \ref{th: vu ngoc th 3.8}); on the other, semi-toric systems are not classified by their associated labeled convex polygons, as subtler symplectic invariants appear (cf. Pelayo $\&$ \vungoc\ \cite{pelayo-vu ngoc inventiones,pelayo_vu_ngoc_con}).\\


Given a semi-toric system $(M, \om, \Phi=(J,H))$, there is an underlying Hamiltonian $S^1$-space $\hamiltonian$ obtained by `forgetting' $H$. This begs the following intriguing question: What is the minimal set of invariants of $\semitoric$ needed to recover the labeled directed graph of $\hamiltonian$? This question has been asked in Pelayo $\&$ \vungoc\ \cite[Remark 6.2]{pelayo-vu ngoc inventiones} and its answer is the main result of the present paper, stated loosely below (cf. Theorem \ref{thm: main} for a precise version).

\begin{mr}
  Any labeled convex polygon associated to a semi-toric system $\semitoric$ yields the labeled directed graph associated to the underlying Hamiltonian $S^1$-space $\hamiltonian$.
\end{mr}

The idea of the proof is to exploit the similarities between symplectic toric manifolds and semi-toric systems. For the former, Karshon \cite{karshon} shows how to recover the labeled directed graph of the associated Hamiltonian $S^1$-space (cf. Remark \ref{rk:delham}). Thus the aim is to mimic Karshon's ideas in the semi-toric case. However, semi-toric systems allow for focus-focus singular points which do not occur in the symplectic toric category; this difficulty is overcome by using (a) the so-called Eliasson-Miranda-Zung local normal form which gives control over the geometry of the system near such singularities (cf.\ Section \ref{sec:str}) and (b) the connectedness of the fibers of $\Phi$, an important fact proved by \vungoc\ \cite{vu ngoc}. \\

Semi-toric systems can be naturally divided in two families: those whose underlying Hamiltonian $S^1$-action can be extended to an effective Hamiltonian $\mathbb{T}^2$-action (cf.\ Definitions \ref{ext} and \ref{def: extendable}), and the rest. The former are called \emph{adaptable}, while the latter are \emph{non-adaptable}. Once the main result is proved, this article turns to obtaining a characterization of adaptable systems, stated below.

\begin{introth}
\mbox{}
\begin{enumerate}[label=(\arabic*), ref=(\arabic*)]
\item \label{item:21} A semi-toric system $(M, \om, \Phi)$ is adaptable if and only if one of its associated labeled convex polygons is Delzant.
\item \label{item:22} Let $\semitoric$ be an adaptable system and denote by $\hamiltonian$ its underlying Hamiltonian $S^1$-space. The family of labeled convex polygons associated to $\semitoric$ contains all Delzant polygons whose corresponding symplectic toric manifolds have $\hamiltonian$ as their associated Hamiltonian $S^1$-space. 
\end{enumerate}
\end{introth}

Part \ref{item:21} is Theorem \ref{smooth}; part \ref{item:22} corresponds to Theorem \ref{prop:smoothvertex} and Corollary \ref{cor:all} and is in fact used to prove one of the two implications of Theorem \ref{smooth}. It generalizes a phenomenon that occurs in all examples of adaptable semi-toric systems in the literature (cf.\ Sadovski{\'{\i}}  $\&$ Z{\^h}ilinski{\'{\i}} \cite{sad_zhi}). The other implication of Theorem \ref{smooth} is obtained by giving a characterization of non-adaptable systems both near fibers containing focus-focus points and globally (cf. Proposition \ref{prop:nad}). Moreover, an explicit example of a non-adaptable system is constructed in Example \ref{example: na}, which, to the best of our knowledge, is the first of its kind. Current work in progress \cite{hss}, inspired by the work of Karshon $\&$ Kessler $\&$ Pinsonnault \cite{kkp}, is concerned with how `wild' the category of semi-toric systems can be.

\subsection*{Organization of the paper}

After the introduction, Section \ref{sec:Hst} recalls the definition and properties of both Hamiltonian $S^1$-spaces and semi-toric systems; many results are quoted with references where to find the proofs and more details.
Section \ref{sec: from VN to K} states and proves the main result of the article. The proof of Theorem \ref{thm: main} is broken down into several steps. Section \ref{sec:adapt-non-adapt} studies adaptable  and non-adaptable semi-toric systems,
which are characterized by Theorem \ref{smooth}, Theorem \ref{prop:smoothvertex}, and Propositions \ref{prop:nad} and \ref{prop:nonsmoothvertex}.

\subsection*{Conventions}
In the whole article, $\sm$ denotes a connected, closed symplectic manifold. Unless otherwise stated, group actions on manifolds are effective, i.e.\ there are no non-trivial elements of the group which act trivially on the whole space. The identification $S^1 = \R/2\pi\Z$ is used throughout.

\subsection*{Acknowledgements}
We would like to thank Tudor Ratiu for suggesting the problem and providing generous funding for meeting at the IST Lisbon and the EPFL. We would like to thank both institutions for their kind hospitality. Moreover, we are indebted to Miguel Abreu, Tudor Ratiu and Margaret Symington for interesting conversations. 

\noindent
D.\ Sepe was supported by FCT postdoctoral fellowship SFRH/BPD/77263/2011.


\section{Hamiltonian $S^1$-spaces and semi-toric systems}

\label{sec:Hst}

Let $(M, \om)$ be a $2n$-dimensional symplectic manifold. 
Since $\om$ is non-degenerate, it induces an isomorphism of vector bundles 
\begin{equation}
  \label{eq:2}
  \begin{split}
    \omega^{\#}: \cotan M &\to \mathrm{T}M \\
    \alpha &\mapsto  X^{\al},
  \end{split}
\end{equation}
\noindent
where $\omega(X^{\al},\cdot) = -\al$. Let $\mathrm{C}^{\infty}(M)$ denote the vector space of smooth functions and let $\diff$ be the exterior differential. The \textbf{Hamiltonian vector field} associated to $F \in \mathrm{C}^{\infty}(M)$ is defined as $X^F = \omega^{\#}(\diff F)$. The Poisson bracket $\{\cdot,\cdot\} : \mathrm{C}^{\infty}(M) \times \mathrm{C}^{\infty}(M) \to \mathrm{C}^{\infty}(M)$ induced by $\omega$ is given by
$$  \{F_1, F_2\}:= \om(X^{F_1}, X^{F_2}). $$

\begin{defin}[{\bf Hamiltonian $\R^k$-actions}]\label{defin:cihs}
A {\em Hamiltonian $\R^k$-action} on $(M, \om)$ is a smooth map $\Phi:=(F_1, \dots, F_k) : M \to \R^k$ satisfying

\begin{enumerate}[label=(\Roman*),ref=(\Roman*)]
\item \label{item:3} $\{F_i, F_j\}=0$ for all $1 \leq i, j \leq k$;
\item \label{item:4} $\diff F_1\wedge \ldots \wedge \diff F_k \neq 0$ almost everywhere.
\end{enumerate}
The triple $(M,\om,\Phi)$ is henceforth referred to as a \emph{Hamiltonian $\R^k$-space}, and $\Phi$ is the {\em moment map}.
\end{defin}

To see that Definition \ref{defin:cihs} yields an $\R^k$-action, let $X^{F_1},\ldots, X^{F_k}$ be the Hamiltonian vector fields associated to $F_1,\ldots,F_k$, and denote by $\varphi_t^1,\ldots,\varphi_t^k$ the corresponding flows. These exist for all $t \in \R$ by compactness of $M$. Moreover, property \ref{item:3} implies that they pairwise commute. Then the $\R^k$-action is given by
\begin{equation*}
 \begin{split}
    \R^k \times M &\to M \\
    (t_1, \dots, t_k)\cdot p&:= \varphi_{t_1}^{1}\circ \cdots \circ \varphi_{t_k}^{k}(p).
  \end{split}
\end{equation*}

Two families of Hamiltonian $\R^k$-spaces play an important role in this paper, namely

\begin{itemize}
\item {\bf completely integrable Hamiltonian systems} when $k=n$ in Definition \ref{defin:cihs}.
\item {\bf Hamiltonian $\T^k$-spaces} if the flows of $X^{F_1},\ldots,X^{F_k}$ are periodic, and the induced torus action
is effective.
\end{itemize}

Henceforth $(M,\omega)$ is taken to be $4$-dimensional, unless otherwise stated.


\subsection{Hamiltonian $S^1$-spaces}\label{def karshon graph}

The aim of this subsection is to introduce Hamiltonian $S^1$-spaces and describe their invariants, as constructed in Karshon \cite{karshon}.

\begin{defin}[\textbf{Hamiltonian $S^1$-spaces}]\label{defn:ham_circle_space}
The category $\Hs$ is defined by:
\begin{itemize}
\item\emph{Objects}: Hamiltonian $S^1$-spaces $(M,\om,J)$.
\item\emph{Morphisms}: symplectomorphisms $\Psi\colon (M_1,\omega_1)\to (M_2,\omega_2)$ making the following diagram
$$ \xymatrix{ (M_1,\omega_1) \ar[rr]^-{\Psi} \ar[dr]_-{J_1} & & (M_2,\omega_2) \ar[dl]^-{J_2} \\
	& \R &	}$$
commute. \\
These are henceforth denoted by $\Psi : (M_1,\omega_1,J_1) \to (M_2,\omega_2,J_2)$
and referred to as isomorphisms of Hamiltonian $S^1$-spaces.
\end{itemize}
\end{defin}

\begin{remark}
  Observe that commutativity of the above diagram implies that the symplectomorphism $\Psi$ is equivariant.
\end{remark}

\begin{example}\label{exm:cp2}
  Consider $\cp$ with homogeneous complex coordinates $[z_0:z_1:z_2]$ and the (standard) Fubini-Study symplectic form $\omega_{FS}$. The map $J : \cp \to \R$ defined by 
\begin{equation*}
 [z_0:z_1:z_2] \mapsto -\frac{1}{2}\Bigg(\frac{|z_1|^2}{|z_0|^2+|z_1|^2+|z_2|^2} +2\frac{|z_2|^2}{|z_0|^2+|z_1|^2+|z_2|^2}\Bigg)
\end{equation*}
\noindent
is the moment map of the following effective Hamiltonian $S^1$-action
$$ \lambda \cdot  [z_0:z_1:z_2] =  [z_0:\lambda z_1:\lambda^2 z_2], $$
\noindent
where $\lambda \in S^1$. Thus the triple $(\cp, \omega_{FS}, J)$ defines an object in $\Hs$.
\end{example} 

A source of interesting examples of Hamiltonian $S^1$-spaces is provided by symplectic toric manifolds, which are defined below.

\begin{defin}[{\bf Symplectic toric manifold}]\label{defin:stm}
 A \emph{symplectic toric manifold} is a 
 Hamiltonian $\T^2$-space $(M,\omega,\mu)$, where $\mu =(\mu_1,\mu_2) : \sm \to \R^2$.
 \end{defin}

\begin{remark}\label{rmk:ext}
Given a symplectic toric manifold, there are several ways to obtain a Hamiltonian $S^1$-space, corresponding to restricting the action to a subgroup $S^1 \subset  \T^2$. Throughout this paper, the triple $(M,\omega,\mu_1)$ is henceforth referred to as the Hamiltonian $S^1$-space \emph{associated} to $(M,\omega,\mu =(\mu_1,\mu_2))$. It is important to notice that not all Hamiltonian $S^1$-spaces arise in this fashion (cf.\ Example \ref{example:ne}).
\end{remark}


\subsection*{Karshon's classification}
The classification of Hamiltonian $S^1$-spaces up to isomorphism has been carried out in Karshon \cite{karshon}, and is recalled
below without proofs in order to introduce ideas and notation used in the rest of the paper. \\

Let $(M,\omega,J)$ be a Hamiltonian $S^1$-space. 
For each subgroup $G\subset S^1$, let $M^G$ be the set of points in $M$ whose stabilizer is $G$.
The connected
components of $M^{S^1}$ are symplectic submanifolds, hence either points or surfaces (since the action is effective); this follows from the following { \bf local normal form}.

\begin{lemma}[Karshon \cite{karshon}, Cor.\ A.7]
\label{lemma: S^1 action normal form}
For each $p\in M^{S^1}$ there exist neighborhoods $U\subset M$ of $p$,  $U_0\subset \C^2$ of $(0,0)$, 
and a symplectomorphism $\Psi\colon(U,\omega)\to (U_0,\omega_0)$, where $\omega_0=\frac{i}{2}(dz_1\wedge d\bar{z}_1+dz_2\wedge d\bar{z}_2)$
making the following diagram commute
$$ \xymatrix{(U,\omega) \ar[rr]^-{\Psi} \ar[rd]_-{J} && (U_0,\omega_0) \ar[dl]^-{J_0} \\
& \R, &} $$ 
\noindent 
with  $J_0(z_1, z_2)= J(p) + \frac{m_1}{2} \abs{z_1}^2 + \frac{m_2}{2} \abs{z_2}^2$.
\end{lemma}
\begin{remark}\label{rmk:w}
Since the action is effective,
the integers $m_1$ and $m_2$ are relatively prime and are called the {\bf isotropy weights} at $p$.
\end{remark}

An important role in the classification of Hamiltonian $S^1$-spaces is played by the subsets which are stabilized by $\Z_k:=\Z \slash k \Z\subset S^1$, the cyclic subgroup of order $k>1$.

\begin{lemma}[Karshon \cite{karshon}, Lemma 2.2]
\label{lemma: Zk spheres}
The closure of each connected component of $M^{\Z_k}$ is a symplectic sphere on which $S^1\slash \Z_k$ acts with two fixed points, which are isolated fixed points in $M^{S^1}$. 
\end{lemma}

Such submanifolds are called {\bf $\Z_k$-spheres}, $k$ being the {\bf isotropy weight}, and the minimum (respectively maximum) of $J$ on a $\Z_k$-sphere is called south (respectively north) pole. 
\\

The work in Karshon \cite{karshon} provides an algorithm which associates a {\bf labeled directed graph} $\Gamma=(V,E)$ to (the isomorphism class of) $(M, \om, J)$:
\begin{longtable}{l p{90mm}} 
{\bf Vertex set $V$:}  &
To every component in $M^{S^1}$ associate a vertex. Those associated to surfaces are drawn as `fat vertices'.
\\
{\bf Labeling of $V$:} &
Label each vertex by the value of $J$ on the corresponding component of the fixed point set. If it is extremal (maximal or minimal), call the vertex {\em extremal} ({\em maximal} or {\em minimal}).

To a fat vertex add the genus of the corresponding surface $\Sigma$ and its normalized symplectic area $\frac{1}{2 \pi} \int_\Sigma \om$ as additional labels.
\\
{\bf Edge set $E$:} &
Every $\Z_k$-sphere gives rise to a directed edge going from its south to its north pole.
\\
{\bf Labeling of $E$:} &
Label each edge with the isotropy weight of the corresponding $\Z_k$-sphere.
\end{longtable}

\begin{remark}\label{rmk:properties}
Not every labeled directed graph arises as the one associated to some $(M,\omega,J)$.
For instance fat vertices can only occur at the minimum or maximum of $J$, and there are
no edges incidents to them (cf.\ Karshon \cite[Section 2.1]{karshon}).
\end{remark}

Such labeled directed graphs classify Hamiltonian $S^1$-spaces up to isomorphism:

\begin{theorem}[Karshon \cite{karshon}, Th.\ 4.1]
\label{th: karshon classification}
Two Hamiltonian $S^1$-spaces are isomorphic if and only if their associated directed labeled graphs are equal.
\end{theorem}

\begin{remark}\label{rmk:grad}
An important role in the proof of Theorem \ref{th: karshon classification} is played by the so-called gradient spheres, whose definition is recalled below. Fix $(M,\omega, J)$ and let $g$ be a compatible metric, i.e.\ an $S^1$-invariant Riemannian metric
such that the endomorphism $\mathcal{J}\colon \mathrm{T}M\to \mathrm{T}M$ defined by $g(u,v)=\omega(u,\mathcal{J}(v))$ is an
almost complex structure. Thus the gradient vector field of the moment map $J$ satisfies
$$
\operatorname{grad}(J)=-\mathcal{J}(X^J).
$$
By invariance of the metric, the flow generated by $\mathcal{J}(X^{J})$ commutes with the circle action, thus obtaining
an $\R \times S^1\simeq \C^{\times}$-action.
The closure of each $\C^\times$-orbit is a topological sphere, called a \emph{gradient sphere}; as above, the minimum (respectively maximum) of $J$ along one such sphere is called the south (respectively north) pole. A gradient sphere 
is \emph{free} if its stabilizer is trivial.
A \emph{chain of gradient spheres} is a sequence $C_1,\ldots,C_l$ of gradient spheres such that the south pole of $C_1$
is a minimum of $J$, the north pole of $C_{i-1}$ coincides with the south pole of $C_i$, for every $i=2,\ldots,l$, and the north
pole of $C_l$ is a maximum for $J$. A chain of gradient spheres is \emph{trivial} if it consists only of one free gradient sphere,
and \emph{non trivial} otherwise.

\end{remark}

In Karshon \cite{karshon} particular attention is given to the relation between Hamiltonian $S^1$-spaces and symplectic toric manifolds. The latter have been classified in Delzant \cite{delzant}, 
where a special role is played by a family of convex polygons defined below.

\begin{defin}[{\bf Delzant polygon}]\label{defin:delzant_polytopes}
  \mbox{}
  \begin{itemize}
  \item A convex polygon $\Delta \subset \R^2$ is {\em simple} if there are exactly $2$ edges meeting at each vertex.
  \item A simple polygon $\Delta$ is {\em rational} if all edges have rational slope, i.e.\ they are subsets of straight lines of the form $\mathbf{x} + s \mathbf{u}_i$ for $\mathbf{x} \in \R^2$, $\mathbf{u}_i \in \Z^2$ primitive, $s \in [0, \infty[$ and $i=1,2$.
  \item A vertex of a simple, rational, convex polygon is {\em smooth} if $\Z\langle \mathbf{u}_1, \mathbf{u}_2\rangle = \Z^2$.
  \end{itemize}
  A simple, rational, convex polygon whose vertices are smooth is said to be \emph{Delzant}.
\end{defin}

Given a symplectic toric manifold $(M,\omega,\mu)$, the image $\mu(M):=\Delta$ is a Delzant polygon and, conversely, any Delzant polygon $\Delta$ determines (up to 
$\T^2$-equivariant symplectomorphisms preserving the moment map) a symplectic toric manifold (cf.\ Delzant \cite{delzant}). A natural question to ask is which Hamiltonian $S^1$-spaces arise as those associated to symplectic toric manifolds (cf.\ Remark \ref{rmk:ext}). To this end, Karshon \cite{karshon} proves the following.

\begin{theorem}[Karshon \cite{karshon}, Prop.\ 5.16 and 5.21]
\label{prop: karshon prop 5.21}
Given a Hamiltonian $S^1$-space $(M,\omega, J)$, the following are equivalent:
\begin{enumerate}[label=(E\arabic*), ref=(E\arabic*)]
\item \label{item:15} The $S^1$-action extends to an effective Hamiltonian 2-torus action with moment map given by $(J,H) : M \to \R^2$, i.e.\ the triple $(M,\omega, (J,H))$ is a symplectic toric manifold.
\item \label{item:17} Each fixed surface has genus zero and each non-extremal level set of $J$ contains at most two non-free orbits.
\item \label{item:18} Each fixed surface has genus zero and there is a compatible metric for which there are no more than two non-trivial chains of gradient spheres.
\end{enumerate}
\end{theorem}

\begin{defin}[{\bf Extendable Hamiltonian $S^1$-spaces}]\label{ext}
A Hamiltonian $S^1$-space $(M,\omega,J)$ is said to be \emph{extendable} if it satisfies any of the conditions of Theorem \ref{prop: karshon prop 5.21}.
\end{defin}

The following theorem of Karshon gives a sufficient condition for a Hamiltonian $S^1$-space to be extendable.

\begin{theorem}[Karshon \cite{karshon}, Th. 5.1]\label{ifp}
Let $(M,\omega,J)$ be a Hamiltonian $S^1$-space such that whose fixed points are isolated. Then $(M,\omega,J)$ comes from a K\"ahler toric variety by restricting the action of the 2-torus to a sub-circle.  
\end{theorem}

\begin{example}\label{example:ne}
There are Hamiltonian $S^1$-spaces which are not extendable (cf.\ Remark \ref{rmk:ext}). For instance, endow $\C \mathbb{P}^1\times \C \mathbb{P}^1$ with the standard symplectic form and consider the $S^1$-action given by
$\lambda\cdot([z_0:z_1],[z_2:z_3])=([z_0:\lambda^mz_1],[z_2:z_3])$, for some fixed $m\in \Z\setminus\{0\}$.
Blow up three points lying on $\{[0:1]\}\times \C \mathbb{P}^1$ by the same amount. These three points are fixed points of the $S^1$-action on the resulting manifold and have the same moment map value. Thus this Hamiltonian $S^1$-space is not extendable by Theorem \ref{prop: karshon prop 5.21}. 
The associated (unlabeled) graph is as in Figure \ref{graph:na}.
\end{example}

\begin{figure}[h]\label{graph:na}
\begin{center}
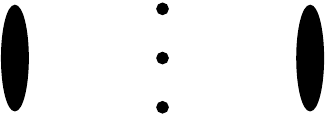
\caption{}
\label{extension}
\end{center}
\end{figure}


\subsection{Semi-toric systems}\label{sec:compl-integr-hamilt}
The aim of this section is to introduce the category of semi-toric systems, to provide some examples, and to describe how to associate a family of polygons to such a system, following \vungoc\ \cite{vu ngoc}. 

\subsubsection{Strongly non-degenerate singularities}\label{sec:str}
Let $(M,\omega,\Phi=(J,H))$ be a completely integrable Hamiltonian system. A point $p \in M$ is \emph{singular} or \emph{critical} if $\Phi$ fails to be a submersion at $p$. In this case, the \emph{rank} of $p$ is defined to be $\mathrm{rk}\, D_{p} \Phi$. Working with arbitrary types of singular points is beyond the scope of this paper. To this end, all singular points are henceforth assumed to be \textbf{non-degenerate} in the sense of Williamson \cite{williamson}, i.e.\ a generalization of the Morse-Bott condition in the symplectic category (cf.\ Zung \cite{zung symplectic} for a precise definition). This notion is generic and naturally extends to singular orbits of the $\R^2$-action, i.e.\ if an orbit $\mathcal{O}$ contains a singular non-degenerate point, then all points in $\mathcal{O}$ are non-degenerate. Moreover, the singular points considered here are not of hyperbolic type, as these are of an intrinsically different nature to the other ones (cf.\ Symington \cite{symington}). These are henceforth referred to as \emph{strongly non-degenerate}. This is no standard notation, introduced here for convenience.
For the purposes of this paper, non-degeneracy amounts to controlling the local behavior of the action near compact singular orbits (cf.\ Eliasson \cite{eliasson thesis}, Miranda $\&$ Zung \cite{miranda-zung}). This can be made precise as follows and is henceforth referred to as the Eliasson-Miranda-Zung \textbf{local normal form}. The above assumptions imply that there are three types of singular orbits, two of rank 0 (i.e.\ fixed points) and one of rank 1 (i.e.\ a circle). \\

\noindent
\underline{Fixed points}: Let $(x,y,\xi,\eta)$ denote Darboux coordinates on $(\R^4,\omega_0)$.\\

\noindent
\textbf{Elliptic-elliptic point}: A point $p \in M$ is said to be of \emph{elliptic-elliptic type} if there exist open neighbourhoods $U \subset (M,\om)$ of $p$, $U_0 \subset (\R^4,\omega_0)$ of $\mathbf{0} \in \R^4$, a symplectomorphism $\Psi: (U,\om) \to (U_0,\om_0)$ such that $\Psi(p) = \mathbf{0}$, and a local diffeomorphism $\psi : \R^2 \to \R^2$ satisfying $\psi(\Phi(p)) = (0,0)$, which make the following diagram commute
  \begin{equation}
    \label{eq:1}
    \xymatrix{ (U,\om) \ar[d]_-{\Phi} \ar[r]^-{\Psi} & (U_0,\om_0) \ar[d]^-{\Phi_{\mathrm{ee}}} \\
      \R^2 \ar[r]_-{\psi} & \R^2,}
  \end{equation}
\noindent
where $\Phi_{\mathrm{ee}} = (q_1,q_2)$ and $q_1 = \frac{1}{2}(x^2+\xi^2)$, $q_2=\frac{1}{2}(y^2+\eta^2)$.\\

\noindent
{\bf Focus-focus point}: A point $p \in M$ is said to be of \emph{focus-focus type} if there exist $U, U_0, \Psi, \psi$ as above making the diagram in equation \eqref{eq:1} commute with respect to the map $\Phi_{\mathrm{ff}} =(q_1,q_2)$, where $q_1 = x\eta - y\xi$, $q_2=x\xi+y\eta$.\\

\noindent
\underline{Rank 1 orbits}: Let $(x,y,a,\theta)$ denote Darboux coordinates on $(\R^2 \times \cotan S^1,\omega_0)$. \\

\noindent
{\bf{Elliptic-regular orbits}}: An orbit $\mathcal{O}$ is said to be of \emph{elliptic-regular type} if there exist open neighbourhoods $U \subset (M,\om)$ of $\mathcal{O}$, $U_0 \subset (\R^2 \times \cotan S^1,\omega_0)$ of the circle $C =\{x=y=a=0\}$, a symplectomorphism $\Psi: (U,\om) \to (U_0,\om_0)$ such that $\Psi(\mathcal{O}) = C$, and a local diffeomorphism $\psi : \R^2 \to \R^2$ satisfying $\psi(\Phi(\mathcal{O})) = (0,0)$, which make the following diagram commute
  \begin{equation*}
    \xymatrix{ (U,\om) \ar[d]_-{\Phi} \ar[r]^-{\Psi} & (U_0,\om_0) \ar[d]^-{\Phi_{\mathrm{er}}} \\
      \R^2 \ar[r]_-{\psi} & \R^2,}
  \end{equation*}
\noindent
where $\Phi_{\mathrm{er}} = (q_1,q_2)$ and $q_1 = \frac{1}{2}(x^2+y^2)$, $q_2= a$. \\

\subsubsection{The category $\st$}\label{sec:st}
With the above definitions in hand, it is now possible to define the category of semi-toric systems.
\begin{defin}[\textbf{Semi-toric systems}, \vungoc\ \cite{vu ngoc}]\label{at&st}
The category $\mathcal{S}\mathcal{T}$ is defined by
\begin{itemize}
\item {\it Objects}: completely integrable Hamiltonian systems $(M,\omega,\Phi=(J,H))$ whose singular points are strongly non-degenerate and such that $(M,\omega, J)$ is a Hamiltonian $S^1$-space. These are henceforth called {\em semi-toric systems}.
\item {\it Morphisms}: pairs $(\Psi, \psi)$, where $\Psi : (M_1,\omega_1) \to (M_2,\omega_2)$ is a symplectomorphism and $\psi : \Phi_1(M_1) \subset \R^2 \to \Phi_2(M_2) \subset \R^2 $ is a locally defined diffeomorphism of the form $\psi(x,y)=(\psi^{(1)}, \psi^{(2)})(x,y) = (x, \psi^{(2)}(x,y))$ making the following diagram commute
$$ \xymatrix{(M_1,\omega_1) \ar[r]^-{\Psi} \ar[d]_-{\Phi_1} & (M_2,\omega_2) \ar[d]^-{\Phi_2} \\
\R^2 \ar[r]_-{\psi} & \R^2.}$$
\end{itemize}
These are henceforth denoted by $(\Psi, \psi) : (M_1,\omega_1,\Phi_1) \to (M_2,\omega_2,\Phi_2)$ and are referred to as isomorphisms of semi-toric systems.
\end{defin}

\begin{remark}\label{rmk:na}
In other works the total space $M$ of a semi-toric system is not necessarily compact, but $J$ is asked to be proper (cf.\ Pelayo $\&$ \vungoc\ \cite{pelayo-vu ngoc inventiones, pelayo_vu_ngoc_con}). 
\end{remark}

Intuitively, semi-toric systems lie at the intersection of completely integrable Hamiltonian systems and Hamiltonian $S^1$-spaces, as formalized by the following remark.
\begin{remark}\label{rmk:fun}
Definitions \ref{defn:ham_circle_space} and \ref{at&st} imply that there is a functor $\mathcal{F} : \st \to \Hs$ defined on objects and morphisms by
\begin{equation*}
\begin{split}
(M,\omega, \Phi =(J,H)) &\mapsto (M,\omega,J) \\
(\Psi,\psi) &\mapsto \Psi.
\end{split}
\end{equation*}
This is a well-defined functor because $\psi$ may only change the second component.
Given $\semitoric$, $\mathcal{F}\semitoric$ is said to be the \textit{underlying} Hamiltonian $S^1$-space.
\end{remark}

\begin{example}\label{exm:stm}
Symplectic toric manifolds (cf.\ Definition \ref{defin:stm}) are, in particular, semi-toric. The only singular orbits of toric systems are either elliptic-elliptic points or elliptic-regular orbits.
\end{example}

\begin{example}\label{exm:sz}
The first examples of honest (i.e.\ with focus-focus points) semi-toric systems appeared in the study of coupled angular momenta carried out in Sadovski{\'{\i}} $\&$ Z{\^h}ilinski{\'{\i}} \cite{sad_zhi}. More generally, the methods of Pelayo $\&$ \vungoc\ \cite{pelayo_vu_ngoc_con} allow to construct semi-toric systems by specifying some initial data.
\end{example}

\begin{remark}\label{rmk:sttoric}
  Semi-toric systems share a very important property with symplectic toric manifolds, namely connectedness of the fibers of the moment map. In the toric category, this fact is used to prove the Atiyah-Guillemin-Sternberg convexity theorem (cf.\ Atiyah \cite{atiyah}, Guillemin $\&$ Sternberg \cite{guillemin-sternberg}), while in the semi-toric case, this follows from \vungoc\ \cite[Theorem 3.4]{vu ngoc}.
\end{remark}

The simplest invariant of the isomorphism class of a semi-toric system $\semitoric$ is the \textbf{number of focus-focus critical points} $m_f \in \mathbb{N} \cup \{0\}$ (cf.\ Pelayo $\&$ \vungoc\ \cite[Lemma 3.2]{pelayo-vu ngoc inventiones}); when compared to symplectic toric manifolds, this is a new invariant.

\subsection*{Semi-toric polygons} In analogy with the case of symplectic toric manifolds, it is possible to associate a family of simple, rational, convex polygons to (the isomorphism class of) a semi-toric system (cf.\ \vungoc\ \cite{vu ngoc}). However, there are two important differences: first, not all vertices need to be smooth, and, second, this family consists of more than one element. These polygons, called \textbf{semi-toric}, play an important role in the proof of the main result of this paper (cf.\ Theorem \ref{thm: main}); as such, their construction is recalled below in some detail (cf.\ \vungoc\ \cite{vu ngoc} for proofs). \\

Throughout this section, fix a semi-toric system $(M,\omega,\Phi)$ with $m_f$ focus-focus critical points. The image $B:=\Phi(M)$ is called the \emph{curved polygon with marked interior points} (often abbreviated to curved polygon) associated to $\semitoric$, where the marked interior points are critical values of $\Phi$ whose fiber contains focus-focus points (see Figure \ref{curvedpolygon}). These are called \emph{focus-focus values} and are denoted by $c_1,c_2,\ldots,c_{m_f}$.

\begin{remark}\label{rmk:mult}
  Note that there may exist $i \neq j$ such that $c_i = c_j$. Any interior marked point $c_i$ is displayed in figures with its {\em multiplicity}, i.e.\ the integer $j(c_i) : = \abs{\{j \in \{1,\ldots,m_f\} \mid c_j = c_i\}}$, which is equal to the number of focus-focus critical points in $\Phi^{-1}(c_i)$, see Figure \ref{curvedpolygon}. 
\end{remark}

Since $B \subset \R^2$, it makes sense to consider the boundary $\partial B := B \setminus \mathring{B}$ (note that $B \subset \R^2$ is closed). A point $s \in\partial B$ is either an \emph{elliptic-elliptic value} (if $\Phi^{-1}(s)$ is an elliptic-elliptic point), or an \emph{elliptic-regular value} (if $\Phi^{-1}(s)$ is an elliptic-regular orbit). The former occur as vertices $B$ as shown in Figure \ref{curvedpolygon}. The segments in $\partial B$ joining vertices are called \textit{curved edges} and consist of elliptic-regular values (except for the vertices). Points in $B_{\mathrm{reg}}:=\mathring{B} \setminus \{c_1,c_2,\ldots,c_{m_f}\}$ are called \emph{regular values} and their fibers are tori. This description follows from connectedness of the fibers and the Eliasson-Miranda-Zung local normal form of Section \ref{sec:str} (cf.\ \vungoc\ \cite{vu ngoc}). 

\begin{remark}\label{rmk:boh}
The curved polygon $B$ has further properties which are proved in \vungoc\ \cite[Theorem 3.4]{vu ngoc}.
\end{remark}

\begin{figure}[h]
\begin{center}
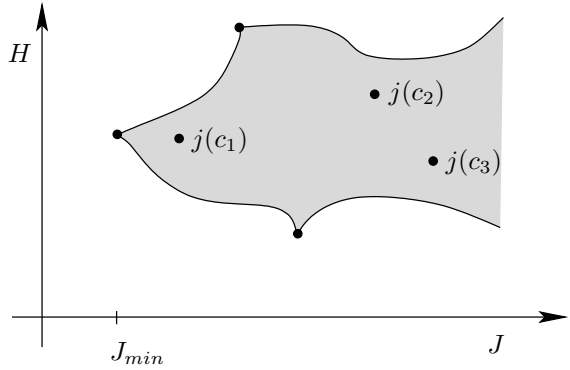
\caption{Part of the curved polygon.}
\label{curvedpolygon}
\end{center}
\end{figure}

The subset $B \setminus \{c_1,c_2,\ldots, c_{m_f}\}$ inherits the structure of a manifold with corners (in the sense of Karshon $\&$ Lerman \cite[Appendix A]{karshon_lerman}) endowed with an integral affine structure, defined below for two dimensional manifolds.

\begin{defin}[{\bf Integral affine structures}]\label{defin:inte}
An \emph{integral affine structure} on $N$ is a smooth atlas $\mathcal{A}:=\{(U_i,\phi_i)\}$, with $\phi_i :U_i \subset N \to \R^2$ such that the change of coordinates $\phi_j \circ \phi_i^{-1}:\phi_i(U_i \cap U_j) \subset \R^2 \to \phi_j(U_i \cap U_j) \subset \R^2$ is an element of $\mathrm{AGL}(2;\Z) := \mathrm{GL}(2;\Z) \ltimes \R^2$.
\end{defin}

The integral affine structure $\mathcal{A}$ on $B$ is defined by taking the \emph{action} coordinates given by the Liouville-Arnol'd theorem near regular values (cf.\ Duistermaat \cite{dui}), and by the Eliasson-Miranda-Zung local normal form of Section  \ref{sec:str} near the boundary.

\begin{remark}\label{rmk:bound}
In integral affine coordinates the boundary $\partial B \subset B$ is locally defined by hyperplanes whose normals have integer coefficients. 
\end{remark}

It is important to note that the integral affine structure $\mathcal{A}$ on $B\setminus \{c_1,c_2,\ldots,c_{m_f}\}$ is not the one coming from the inclusion $B \subset \R^2$ unless $m_f =0$ (cf.\ Zung \cite{zung focus focus}). In order to bypass this issue, \vungoc\ \cite{vu ngoc} introduces vertical cuts on $B$ in such a way that the resulting subset is a simply connected integral affine manifold with corners. These can be defined as follows. Let $p_1,\ldots, p_{m_f} \in M$ be the focus-focus singular points of $\semitoric$ and order the corresponding focus-focus values $c_1=(x_1, y_1), \ldots ,c_{m_f}=(x_{m_f}, y_{m_f}) \in \R^2$ so that $x_1 \leq \dots \leq x_{m_f}$. Note that it may be possible that $c_i = c_j$ for some $i \neq j$ (cf.\ Remark \ref{rmk:mult}). For $\ep_i \in \{+1,-1\}$, set
$$l_i^{\ep_i}:= \{(x,y)\in \R^2\mid x=x_i,\; \ep_iy \geq \ep_iy_i\} \cap B.$$ 
\noindent
For $\ep_i =1$ (respectively $-1$) this is the closed vertical segment between $c_i$ and the upper (respectively lower) boundary of $B$. For $\boldsymbol{\ep} = (\ep_1,\ldots,\ep_{m_f}) \in \{+1,-1\}^{m_f}$, set 
$$l^{\boldsymbol{\ep}}=\bigcup\limits_{i=1}^{m_f} l_i^{\ep_i};$$
\noindent
denote  the \emph{open} vertical segments by
$$\mathring{l}^{\ep_i}_i := l^{\ep_i}_i \setminus (\{c_i\} \cup (l^{\ep_i}_i \cap \partial B)).$$
\noindent
Each point $s \in \mathring{l}^{\boldsymbol{\ep}}$ is labeled by the integer
$$ j(s) :=\sum\limits_{i\,\, with \,\, s \,\in\, l^{\ep_i}_i} \ep_i j(c_i), $$ 
\noindent
where $j(c_i)$ is the multiplicity of the focus-focus critical value $c_i$ (cf.\ Remark \ref{rmk:mult}). A choice of cuts $\boldsymbol{\ep}$ determines a labeled convex polygon $\mathcal{P}_{\boldsymbol{\ep}}$ (these are henceforth called semi-toric) associated to $(M,\omega,\Phi)$ in the following fashion.

\begin{theorem}[\vungoc\ \cite{vu ngoc}, Th.\ 3.8]
\label{th: vu ngoc th 3.8}
For any $\boldsymbol{\ep} \in \{ +1, -1\} ^{m_f}$ there exists a homeomorphism $f: B \to f(B)=:\mathcal{P}_{\boldsymbol{\ep}} \subset \R^2$ such that
\begin{enumerate}[label=(\arabic*), ref=(\arabic*)]
  \item \label{item:7}
$f$ restricted to $B \setminus l^{\boldsymbol{\ep}}$ is a diffeomorphism onto its image.
  \item \label{item:8}
$f$ restricted to $B \setminus l^{\boldsymbol{\ep}}$ is integral affine for the standard integral affine structure on $\R^2$.
  \item \label{item:9}
$f$ is of the form $f(x, y)= (x, f^{(2)}(x, y))$.
  \item \label{item:10}
For all $1 \le i \leq m_f$ and all points $s \in \mathring{l}_i^{\ep_i}$, there is an open ball $D$ around $s$ such that the restriction of $f$ to $B\backslash l^{\boldsymbol{\ep}}$ has a smooth extension to $\{(x, y) \in D \mid x \leq x_i\}$ and $\{(x, y) \in D \mid x \geq x_i\}$ and 
\begin{equation*}
  \lim_{\stackrel{(x, y) \to s}{x <x_i}}d f(x, y)= 
  \left(
  \begin{smallmatrix}
   1 & 0 \\ j(s) & 1
  \end{smallmatrix}
  \right)
  \lim_{\stackrel{(x, y) \to s}{x >x_i}}d f(x, y).
\end{equation*}
  \item \label{item:11}
$\mathcal{P}_{\boldsymbol{\ep}}$ is a simple, rational, convex polygon.
\end{enumerate}
\end{theorem}

\begin{remark}\label{rmk:label}
  The labeling of $\polygon$ consists of marking the image $f(c_i)$ of each focus-focus critical value $c_i \in B$ with its multiplicity $j(c_i)$ (cf.\ Remark \ref{rmk:mult}), as illustrated in Figure \ref{straightenedpolygon}.
\end{remark}

\begin{remark}\label{rmk:tau}
In fact, any other semi-toric polygon associated to the same choice of cuts $\boldsymbol{\ep}$ differs from $\mathcal{P}_{\boldsymbol{\ep}}$ by composition with an element of 
$$ \msT := \{\left( \left(
  \begin{smallmatrix}
  1 & 0 \\ j & 1
  \end{smallmatrix}
  \right),
  \left(
  \begin{smallmatrix}
   0 \\ t
  \end{smallmatrix}
  \right)
\right) \mid j \in \Z\,,\, t \in \R\} \subset \mathrm{AGL}(2;\Z),$$
\noindent
i.e.\ the subgroup of integral affine transformations preserving vertical lines. This is because, once a choice of cuts $\boldsymbol{\ep} \in \{+1,-1\}^{m_f}$ is fixed, the homeomorphism $f$ of Theorem \ref{th: vu ngoc th 3.8} is completely determined by a specific choice of action-angle variables near a regular level set of $\Phi$ (cf.\ Pelayo $\&$ \vungoc\ \cite[Section 2.2]{pelayo_vu_ngoc_con} and \vungoc\ \cite[Step 2, Th.\ 3.8]{vu ngoc}). The adjective `specific' refers to the fact that the first component $J$ of $\Phi$ is chosen as an action coordinate (equivalently, the first standard coordinate $x$ on $\R^2$ is chosen as an integral affine coordinate on $B$), since it generates an effective Hamiltonian $S^1$-action. Moreover, upon choosing an orientation on $\R^2$, $f$ can always be chosen so that the top (respectively bottom) boundary of $B$ is sent to the top (respectively bottom) boundary of $\polygon$ by changing the sign of its second component. Henceforth, whenever referring to the semi-toric polygon $\mathcal{P}_{\boldsymbol{\ep}}$ associated to $\semitoric$ and $\boldsymbol{\ep} \in \{+1,-1\}^{m_f}$, it is understood that a choice of action variables (equivalently local integral affine coordinates on $B$) as above is fixed and that $f$ is chosen to be orientation preserving (upon a choice of orientation on $\R^2$), unless otherwise stated.
\end{remark}

\begin{remark}\label{rmk:choice}
  Let $\boldsymbol{\ep}, \boldsymbol{\ep}'$ be two choices of cuts for $\semitoric$ and denote the corresponding semi-toric polygons by $\polygon$, $\mathcal{P}_{\boldsymbol{\ep}'}$. Then there exists a continuous piecewise integral affine transformation $\tau$ such that $\mathcal{P}_{\boldsymbol{\ep}'} = \tau(\polygon)$ with the property that $\tau$ preserves vertical lines, i.e.\ on each region on which it is defined by an integral affine transformation it is given by a restriction of an element in $\msT$. This can be used to give a geometric interpretation of the action of $\{+1,-1\}^{m_f}$ on the space of semi-toric polygons associated to $\semitoric$ (cf.\ \vungoc\ \cite[Prop.\ 4.1]{vu ngoc}).
\end{remark}

\begin{figure}[h]
\begin{center}
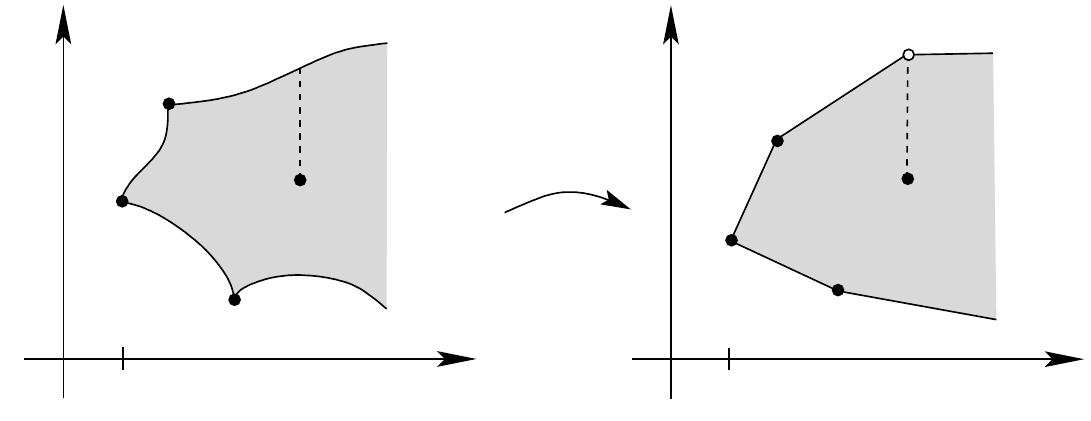
\caption{From $B = \Phi(M)$ to $\polygon$}
\label{straightenedpolygon}
\end{center}
\end{figure}
Fix a choice of cuts $\boldsymbol{\ep} \in \{+1,-1\}^{m_f}$ and let $\mathcal{P}_{\boldsymbol{\ep}}$ be the associated semi-toric polygon. The vertices of $\polygon$ fall into three categories, as described below.

\begin{defin}[{\bf Types of vertices of $\polygon$}, Def. 4.1 in Pelayo $\&$ \vungoc\ \cite{pelayo_vu_ngoc_con}]\label{defn:vertices}

A vertex $v$ of $\mathcal{P}_{\boldsymbol{\ep}}$ is said to be
\begin{itemize}
\item \emph{Delzant} if it is the image of a vertex of $B$ which has no cuts `into' it;
\item \emph{hidden Delzant} if it is the image of a vertex of $B$ which has at least one cut `into' it;
\item \emph{fake}: if it is the image of a point in $\partial B$ which is not a vertex.
\end{itemize}
For a hidden Delzant or fake vertex $v$, its \emph{degree} $n_v \geq 1$ is the number of cuts that go into $f^{-1}(v)$, while its \emph{sign} $\ep_v \in \{+1,-1\}$ is the sign of any cut going into $f^{-1}(v)$. 
\end{defin}

\begin{remark}\label{rmk:nv}
Henceforth, Delzant, hidden Delzant and fake vertices are displayed in figures by $\bullet$, $\odot$ and $\circ$ respectively.
\end{remark}

\begin{remark}\label{rmk:pv}
  Note that the preimage $(f \circ \Phi)^{-1}(v)$ is an elliptic-elliptic point if (and only if) $v$ is Delzant or hidden Delzant, while it is a circle (consisting of elliptic-regular points) otherwise.
\end{remark}

Let $J_{\mathrm{min}}$ (respectively $J_{\mathrm{max}}$) denote the minimum (respectively maximum) value taken by $J$. By construction of $f$, vertices lying on $\polygon \cap \{(x,y) \mid x = J_{\mathrm{min}}\}$ (respectively $J_{\mathrm{max}}$) are Delzant (cf.\ \vungoc\ \cite[proof of Th.\ 3.8]{vu ngoc}). Fix a vertex $v$ of $\polygon$ strictly between the vertical lines between $J_{\mathrm{min}}$ and $J_{\mathrm{max}}$. By Theorem \ref{th: vu ngoc th 3.8}, the edges incident to $v$ have integral tangent vectors; denote their primitives with positive first component by $\mathbf{u},\mathbf{w} \in \Z^2$. Throughout, $\mathbf{u}$ (respectively $\mathbf{w}$) denotes the primitive tangent to the edge on the `left' (respectively `right') of $v$ (the orientation is chosen so that $J$ does not decrease going from left to right). As mentioned above, $\Z \langle \mathbf{u},\mathbf{w} \rangle$ does not need to be the standard $\Z^2 \subset \R^2$, i.e.\ the vertex does not need to be smooth. However, using the results of \vungoc\ \cite{vu ngoc}, the lemma below proves some conditions on $\mathbf{u}, \mathbf{w}$, which generalise Pelayo $\&$ \vungoc\ \cite[Def.\ 4.1 $\&$ Lemma 4.4]{pelayo_vu_ngoc_con}.

\begin{lemma}\label{lemma:df}
  If the vertex $v$ is
  \begin{itemize}
  \item Delzant, then $ \Z \langle \mathbf{u},\mathbf{w} \rangle = \Z^2$,
  \item hidden Delzant, then $ \Z \langle \mathbf{u},A_v\mathbf{w} \rangle = \Z^2$,
  \item fake, then $\Z \langle \mathbf{u},A_v\mathbf{w} \rangle = \Z \langle \mathbf{u} \rangle$,
  \end{itemize}
  where 
  $$ A_v :=
  \begin{pmatrix}
    1 & 0 \\
    \ep_v n_v & 1
  \end{pmatrix}, $$
  \noindent
  $n_v,\ep_v$ being the degree and sign of $v$ respectively. 
\end{lemma}
\begin{proof}
  If $v$ is Delzant, then $f \circ \Phi$ defines a Hamiltonian $\T^2$-action near $(f \circ \Phi)^{-1}(v)$ which commutes with the Hamiltonian action defined by $\Phi$. Thus, in this case, the result follows because the action is locally toric. If $v$ is either hidden Delzant or fake, consider a semi-toric polygon $\mathcal{P}_{\boldsymbol{\ep}'}$ associated to the choice of cuts $\boldsymbol{\ep'}$, which agrees with $\boldsymbol{\ep}$ except that it changes the sign of the cuts going into $f^{-1}(v)$. Let $f': B \to \mathcal{P}_{\boldsymbol{\ep}'}$ denote the homeomorphism associated to $\boldsymbol{\ep}'$ as in Theorem \ref{th: vu ngoc th 3.8}. Then, by construction of $f'$, we have $f'(f^{-1}(v)) = v$. If $v$ was hidden Delzant for $\polygon$, then $v$ is Delzant for $\mathcal{P}_{\boldsymbol{\ep}'}$, while if it was fake for $\polygon$, it is not a vertex for $\mathcal{P}_{\boldsymbol{\ep}'}$. The result in both cases follows from the fact that $\mathcal{P}_{\boldsymbol{\ep}'} = \tau_v(\polygon)$, where $\tau_v$ is a piecewise integral affine transformation, which is the identity on the half-space to the left of the vertical line containing $v$ and $A_v$ on the half-space to the right (cf.\ Remark \ref{rmk:choice} and \vungoc\ \cite[Prop.\ 4.1]{vu ngoc}).
\end{proof}

Any semi-toric polygon $\polygon$ associated to $(M,\omega,\Phi=(J,H))$ contains information about an important invariant of the underlying Hamiltonian $S^1$-space $\hamiltonian$.

\begin{defin}[{\bf Duistermaat-Heckman measure}, \cite{dui_heck}]\label{defin:dhm}
  The push forward of the Liouville measure $\om^2$ under $J$ is called the {\em Duistermaat-Heckman measure} $\mu_J$ of the Hamiltonian $S^1$-action defined by $J$. Its density function $\rho_J$ is called the associated {\em Duistermaat-Heckman function}.
\end{defin}

In analogy with symplectic toric manifolds, $\polygon$ determines the derivative of the Duistermaat-Heckman function of $\hamiltonian$ as stated in the theorem below. This is an important fact which is used in the proof of the main result.

\begin{theorem}
[\vungoc\ \cite{vu ngoc}, Th.\ 5.3]
\label{th: vu ngoc th 5.3}
Let $x \in \R$. If $x$ is not a critical value of $J$, then the derivative of the Duistermaat-Heckman function is given by 
\begin{equation*}
\rho'_J(x)=\alpha^+(x) - \alpha^-(x),
\end{equation*}
where $\alpha^{\pm}(x)$ denotes the slope of the top (respectively bottom) edge of $\polygon$ intersecting the vertical line through $x$. Otherwise, 
\begin{equation}
 \label{eq:weightformula}
\rho'_J(x+0) - \rho'_j(x-0 ) = -e^+ - e^- - j_x 
\end{equation}
where $e^+$ (respectively $e^-$) is zero if there is no top (respectively bottom) vertex whose first coordinate is $x$, or $e^\pm = -\frac{1}{a^\pm b^\pm}$, where $a^{\pm},b^{\pm}$ are the isotropy weights of the $S^1$-action at the corresponding vertices, and $j_x$ is the number of focus-focus critical points of $\Phi$ lying in $J^{-1}(x) \subset M$.
\end{theorem}


\section{The main theorem: from semi-toric polygons to labeled directed graphs}\label{sec: from VN to K}

Recall that there is a functor $\mathcal{F}: \st \to \Hs$, which sends a semi-toric system $(M,\omega,\Phi=(J,H))$ to its underlying Hamiltonian $S^1$-space $\hamiltonian$ (cf.\ Remark \ref{rmk:fun}). A natural question to ask is to describe how to recover the invariants of (the isomorphism class of) a Hamiltonian $S^1$-space underlying (the isomorphism class of) a semi-toric system from the invariants of the latter. On the one hand, isomorphism classes of Hamiltonian $S^1$-spaces are classified by their associated labeled directed graphs (cf.\ Section \ref{def karshon graph} and Karshon \cite{karshon}); on the other, there is no theorem classifying isomorphism classes of semi-toric systems in full generality (cf.\ Pelayo $\&$ \vungoc\ \cite{pelayo-vu ngoc inventiones,pelayo_vu_ngoc_con} for the classification of generic semi-toric systems). However, all that is needed to recover the labeled directed graphs of Hamiltonian $S^1$-spaces underlying semi-toric systems are the number of focus-focus critical points and the associated semi-toric polygons introduced in Section \ref{sec:compl-integr-hamilt}: this is the content of the main theorem of this paper, stated below.

\begin{theorem}\label{thm: main}
Let $\semitoric$ be a semi-toric system with $m_f$ focus-focus critical points, and let $\hamiltonian$ denote the underlying Hamiltonian $S^1$-space. For any choice of cuts $\boldsymbol{\ep} \in \{+1,-1\}^{m_f}$, the associated semi-toric polygon $\polygon$ and $m_f$ determine the labeled directed graph $\Gamma$, thus classifying $\hamiltonian$ up to isomorphisms in the $\Hs$ category. 
\end{theorem}

Throughout this section, fix a semi-toric system $\semitoric$ along with a semi-toric polygon $\polygon$ and $m_f$ focus-focus critical points, and denote its underlying Hamiltonian $S^1$-space by $\hamiltonian$ unless otherwise stated. This automatically sets the homeomorphism $f: B =\Phi(M) \to \polygon$ given by Theorem \ref{th: vu ngoc th 3.8}. Recall that the labeled directed graph $\Gamma$ associated to $\hamiltonian$ is determined by the \textbf{vertex set} $V$ and its labeling (i.e.\ the connected components of the fixed point set $M^{S^1}$, and their topological and symplectic properties respectively), and the {\bf edge set} $E$ and its labeling (i.e.\ $\Z_k$-spheres). Recovering $V$ and its labeling from $\polygon$ is the aim of Section \ref{sec:vertex-set-v}, while Section \ref{sec:edge-set-e} deals with `seeing' $\Z_k$-spheres from $\polygon$. The proof of Theorem \ref{thm: main} is then given in Section \ref{sec:proof-theor-refthm}, which brings everything together. \\

The scheme of the proof of Theorem \ref{thm: main} mimics the relation between symplectic toric manifolds and their associated Hamiltonian $S^1$-spaces, which is recalled below.

\begin{remark}[Karshon \cite{karshon}, Section 2.2]\label{rk:delham}
The following table shows how to pass from a Delzant polygon $\Delta$ associated to a symplectic toric manifold $(M,\omega,\mu)$ to the labeled directed graph $\Gamma$ of its associated Hamiltonian $S^1$-space. This serves as a guide to follow the ideas of the forthcoming sections.
\end{remark}

\begin{longtable}{l p{90mm}}
{\bf Vertex set V:} & Each vertex $v_{\Delta}$ of $\Delta$ which is not incident to a vertical edge corresponds to a vertex $v_{\Gamma}$ of $\Gamma$. A vertical edge $e^{\mathrm{vert}}_{\Delta}$ of $\Delta$ gives rise to a fat vertex $v^{\mathrm{fat}}_{\Gamma}$ of $\Gamma$.\\
{\bf Labeling of V:} &  Each vertex $v_{\Gamma}$ (respectively $v^{\mathrm{fat}}_{\Gamma}$) in $V$ is labeled with the value of the first coordinate of the corresponding vertex $v_{\Delta}$ (respectively of the corresponding vertical edge $e^{\mathrm{vert}}_{\Delta}$) in $\Delta$. Fat vertices are also labeled with 0 for the genus of the corresponding fixed surfaces (they are all spheres) and with the length of the corresponding vertical edge of $\Delta$ for its normalised symplectic area. \\
{\bf Edge set E:} & An edge $e_{\Delta}$ of $\Delta$ whose primitive tangent vector is of the form $(k,b) \in \Z^2$ (for $k \geq 2$) gives rise to an edge $e_{\Gamma}$ in $\Gamma$ joining the vertices in $V$ corresponding to the vertices of $\Delta$ of $e_{\Delta}$ (note that these vertices can never be fat). \\
{\bf Labeling of E:} & Each edge $e_{\Gamma}$ of $\Gamma$ is labeled with the integer $k \geq 2$, where $(k,b) \in \Z^2$ is a primitive tangent vector to the corresponding edge $e_{\Delta}$ of $\Delta$.
\end{longtable}


\subsection{Fixed point set $M^{S^1}$: vertex set $V$ and its labeling}\label{sec:vertex-set-v}
This section describes how to construct and label the vertex set $V$ of $\Gamma$ from the fixed data $m_f$ and $\polygon$. By definition, vertices of $\Gamma$ correspond to connected components of $M^{S^1}$, the fixed point set of the $S^1$-action whose moment map is $J$. Any point $p \in M$ which is fixed by this $S^1$-action satisfies $\diff J(p) =0$. Thus $p$ is also a critical point of $\Phi$. The next lemma characterizes isolated fixed points of the $S^1$-action in terms of critical points of $\Phi$.

\begin{lemma}\label{lemma:vertices}
  The isolated fixed points in $M^{S^1}$ are either
  \begin{enumerate}[label=(V\arabic*), ref=(V\arabic*)]
  \item \label{item:5} focus-focus critical points of $\Phi$, or
  \item \label{item:6} elliptic-elliptic critical points of $\Phi$ whose image in $\polygon$ is not a vertex of a vertical edge.
  \end{enumerate}
\end{lemma}
\begin{proof}
  This uses the Eliasson-Miranda-Zung local normal form (cf.\ Section \ref{sec:str}). First, observe that any isolated fixed point $p \in M^{S^1}$ cannot be of elliptic-regular type, since, if so, all points in $\Phi^{-1}(\Phi(p))$ (= the orbit of $p$ of the Hamiltonian $\R^2$-action whose moment map is $\Phi$) have the same stabilizer. Thus $p$ needs to satisfy $\mathrm{rk} \, D_p\Phi =0$. It follows again from the Eliasson-Miranda-Zung local normal form theorem that the isolated fixed points in $M^{S^1}$ are precisely rank 0 points satisfying conditions \ref{item:5} and \ref{item:6}.
\end{proof}

Having dealt with isolated fixed points in $M^{S^1}$, consider the fixed surfaces, which, by the following proposition, are precisely as in the case of symplectic toric manifolds (cf.\ Remark \ref{rk:delham}).

\begin{prop}
\label{prop: fixed surf}
Let $\Sigma$ be a connected surface which is fixed by the $S^1$-action. Then either $\Sigma=J^{-1}(J_{\mathrm{min}})$ or $\Sigma=J^{-1}(J_{\mathrm{max}})$, where $J_{\mathrm{min}}$ (respectively $J_{\mathrm{max}}$) is the minimum (respectively maximum) value of $J$ on $M$. Moreover, $\Sigma$ is a symplectic sphere. 
\end{prop}
\begin{proof}
By a standard argument which uses local normal forms (cf.\ Karshon \cite[Cor.\ A.7]{karshon}), 
$\Sigma$ is a symplectic surface which is a local minimum or maximum  for $J$.
Since $J$ is a moment map for an $S^1$-action, by the Atiyah-Guillemin-Sternberg convexity theorem,
a local minimum or maximum is global, and is connected (cf.\ Atiyah \cite{atiyah}, Guillemin $\&$ Sternberg \cite{guillemin-sternberg}).
From this it follows that $\Sigma=J^{-1}(J_{\mathrm{min}})$ or $\Sigma=J^{-1}(J_{\mathrm{max}})$.

It remains to prove that $\Sigma$ is a sphere. Without loss of generality suppose that $\Sigma=J^{-1}(J_{\mathrm{min}})$. Then there exists $x_0>J_{\mathrm{min}}$ such that $\Sigma=M^{S^1}\cap J^{-1}([J_{\mathrm{min}},x_0[)$, $A=\Phi(J^{-1}([J_{\mathrm{min}},x_0[))\subset \R^2$ is simply connected, and $A$ contains no focus-focus critical values (this follows from the Eliasson-Miranda-Zung local normal form and connectedness of the fibers of $\Phi$). Thus there exist global action-angle coordinates on $J^{-1}([J_{\mathrm{min}},x_0[)$ such that $J$ can be taken to be the first action coordinate (cf.\ \vungoc\ \cite[Prop. 2.12]{vu ngoc}). In other words, $f \circ \Phi(\Sigma)$ corresponds to a vertical edge of $\polygon$ and, near $\Sigma$, the second component of $f \circ \Phi$ is the moment map of an effective Hamiltonian $S^1$-action, which can be restricted to $\Sigma$. Hence $\Sigma$ is a symplectic surface with an effective Hamiltonian circle action, which implies that it is a sphere. 
\end{proof}

\begin{remark}\label{rmk:pfsigma}
  The arguments in the proof of Proposition \ref{prop: fixed surf} imply that $f \circ \Phi(\Sigma) \subset \polygon$ is a vertical edge, and its normalized symplectic area is the length of the corresponding vertical edge, just as in the symplectic toric case.
\end{remark}

\begin{remark}\label{rmk:kir}
  The Eliasson-Miranda-Zung local normal form implies that no focus-focus point lies on $J^{-1}(J_{\mathrm{min}})$ or on $J^{-1}(J_{\mathrm{max}})$.
  It is well-known that $J$ is an $S^1$-invariant Morse function, with critical points
  equal to the fixed points of the $S^1$ action. Thus, by the previous argument,
  the Morse index of $J$ at any such point is 2. (Note that the Morse indices of an $S^1$-invariant Morse function are always even.)
  Following Kirwan \cite{kirwan}, this allows to place the following bound on $m_f$,
$$ m_f \leq \mathrm{rk}\,\mathrm{H}^2(M;\Z), $$
\noindent
where the equality holds if and only if 
there are no fixed surfaces at the minimum of maximum of $J$, and no elliptic-elliptic points in $J^{-1}(]J_{\mathrm{min}},J_{\mathrm{max}}[)$.

Moreover, Proposition \ref{prop: fixed surf} implies that $\mathrm{H}^{\mathrm{odd}}(M;\Z) = 0$, as all fixed surfaces are simply connected.
\end{remark}


\subsection{$\Z_k$-spheres: edge set $E$ and its labeling}\label{sec:edge-set-e} 
Recall that the edges in the labeled directed graph $\Gamma$ associated to $\hamiltonian$ correspond to symplectic spheres in $M$ which are stabilized by a finite subgroup $\Z_k \subset S^1$ with $k \geq 2$; these are known as $\Z_k$-spheres. Fix one such symplectic sphere $\Sigma$; the action of $S^1$ on $\Sigma$  has two fixed points (called the poles of the sphere), which are isolated fixed points in $M^{S^1}$ (cf.\ Karshon \cite{karshon}). An isolated fixed point in $M^{S^1}$ is a pole of a $\Z_k$-sphere if and only if one of its isotropy weights equals $k$ in absolute value (cf.\ Remark \ref{rmk:w}). By Lemma \ref{lemma:vertices}, the isolated fixed points in $M^{S^1}$ are either focus-focus or elliptic-elliptic critical points  satisfying property \ref{item:6}. Thus, before trying to `see' $\Z_k$-spheres from $\polygon$, it is necessary to understand how to obtain the isotropy weights of isolated fixed points in $M^{S^1}$ from $\polygon$. This is the aim of the next subsection.

\subsubsection{Isotropy weights from $\polygon$}\label{sec:isotr-weights-focus}
Lemma \ref{lemma:vertices} proves that isolated fixed points in $M^{S^1}$ satisfy either property \ref{item:5}, i.e.\ they are of focus-focus type, or \ref{item:6}, i.e.\ they are of elliptic-elliptic type with an extra condition. Note that the former do not arise when considering symplectic toric manifolds and, as such, need to be dealt differently. To this end, each of the two cases is discussed separately below.\\

Let $p \in M^{S^1}$ be a focus-focus critical point for $\Phi$. Recall that the Eliasson-Miranda-Zung local normal form gives
\begin{itemize}
\item open neighbourhoods $U \subset M$ of $p$, $U_0 \subset \R^4$ of $\mathbf{0}$;
\item a symplectomorphism $\Psi : (U,\omega) \to (U_0,\omega_0)$, where $\omega_0$ is the standard symplectic form on $\R^4$, and a local diffeomorphism $\psi : \R^2 \to \R^2$ satisfying $\Psi(p)=\mathbf{0}$, $\psi(\Phi(p)) = (0,0)$;
\end{itemize}
which make the following diagram commute
$$ \xymatrix{(U,\omega) \ar[r]^-{\Psi} \ar[d]_-{\Phi} & (U_0,\omega_0) \ar[d]^-{\Phi_{\mathrm{ff}}} \\
  \R^2 \ar[r]_-{\psi} & \R^2,} $$
\noindent
where $\Phi_{\mathrm{ff}} =(q_1,q_2)$, $q_1 = x \eta - y \xi$, $q_2 = x\xi + y \eta$, and $\omega_0 = \diff x \wedge \diff \xi + \diff y \wedge \diff \eta$.

\begin{definition}[{\bf Local system preserving actions}]
An $S^1$-action on $U$ is said to be {\em local system preserving} if for all $\lam \in S^1$ and all $ p \in U$, $\Phi(\lam \cdot p)=\Phi(p)$.
\end{definition}

\begin{remark}\label{rmk:lssa}
  Local system preserving actions play an important role in the topological and symplectic classification of completely integrable Hamiltonian systems developed in Zung \cite{zung symplectic}. 
\end{remark}

The following proposition (known to experts and included here for completeness) shows that there exists a {\em unique}  (up to sign) effective local system preserving Hamiltonian $S^1$-action near a focus-focus singular point (cf.\ Zung \cite[Theorem 1.2]{zung_another} for a different proof).

\begin{prop} \label{prop:unique_system_pres_action}
  There exists a unique (up to sign) local system preserving effective Hamiltonian $S^1$-action defined in a neighbourhood of a focus-focus point.
\end{prop}

\begin{proof}
  Using the Eliasson-Miranda-Zung local normal form, it suffices to prove the result for the linear model $(U_0,\omega_0,\Phi_{\mathrm{ff}} =(q_1,q_2))$ described above. Let $h : (\R^4,\om_0) \to \R$ be the momentum map for an effective Hamiltonian $S^1$-action which is system preserving. By definition, for $i =1,2$,
  \begin{equation}
    \label{eq:3}
    \diff q_i (X^h) = 0,
  \end{equation}
  \noindent
  where $X^h$ denotes the Hamiltonian vector field of $h$. Equation \eqref{eq:3} has the following two consequences:
  \begin{enumerate}[label=(\arabic*), ref=(\arabic*)]
  \item \label{item:1} For all $z \in \R^4$, $X^h(z) \in \ker D_z \Phi_0$.
  \item \label{item:2} $\{q_i,h\}_0 = 0$, where $\{\cdot,\cdot \}_0$ is the Poisson bracket induced by $\om_0$ on $U_0$.
  \end{enumerate}
  Since for all $z \in \R^4 \setminus \{\mathbf{0}\}$, $\ker D_z \Phi_0 = \langle X^{q_1}(z),X^{q_2}(z) \rangle$, \ref{item:1} implies that there exist smooth functions $F_1,F_2 : U_0 \setminus \{\mathbf{0}\} \to \R $ such that, for all $ z\in U_0 \setminus \{\mathbf{0}\}$,
  \begin{equation}
    \label{eq:4}
    X^h(z) = F_1(z) X^{q_1}(z) + F_2(z) X^{q_2}(z).
  \end{equation}
  \noindent
  Consequence \ref{item:2} implies that, for $i=1,2$,
  $$ [X^{q_i}, X^h] = 0. $$
  \noindent
  By equation \eqref{eq:4} and $[X^{q_1},X^{q_2}]=0$, for $i=1,2$ and for all $z \in U_0 \setminus \{\mathbf{0}\}$, the following holds
  \begin{align*}
    0 & = [X^{q_i}, X^h](z) = [X^{q_i}, F_1 X^{q_1} + F_2 X^{q_2}](z) \\ 
      & = ((X^{q_i}F_1)(z))X^{q_1}(z) + ((X^{q_i}F_2)(z))X^{q_2}(z).
  \end{align*}
  \noindent
  Therefore, for $i,j=1,2$ and for all $z \in U_0 \setminus \{0\}$, 
  $$ (X^{q_i} F_j)(z) = 0.$$
  \noindent
  The above equation implies that the functions $F_1,F_2$ are \emph{basic}, i.e.\ there exist smooth functions $G_1,G_2 :\R^2 \setminus \{(0,0)\} \to \R$ such that $F_j = \Phi^*_{\mathrm{ff}} G_j$ for $j=1,2$. 

  Consider the flow of $X^h$, which is periodic with period $2\pi$. Since the functions $F_j$ are basic, they are constant along orbits of $X^h$, as they only depend on the values on the image of $\Phi_0$ and the latter is constant on the orbits of $X^h$. Using the fact that $[X^{q_1},X^{q_2}] = 0$, it therefore follows that the flow of $X^h$ is given by
  \begin{equation}
    \label{eq:5}
    \varphi^{t}_h (w_1,w_2) = (e^{(iF_1+F_2)t}w_1,e^{(iF_1-F_2)t}w_2),
  \end{equation}
  \noindent
  where $w_1 = x + i \xi$, $w_2 = y + i \eta$, and $F_1,F_2$ are smooth functions of $z_1,z_2$ (in fact, of $q_1,q_2$). Since $\varphi^{2\pi}_h = \mathrm{id}$, \eqref{eq:5} implies that
  $$ e^{(i F_1 + F_2)2\pi} = 1 \qquad e^{(i F_1 - F_2)2\pi} = 1, $$
  \noindent
  which implies that $F_2 \equiv 0$ and that $2\pi F_1 \in 2\pi \Z$. Since the $S^1$-action is effective, it follows that $ \abs{F_1} \equiv 1$. Thus, up to sign, $h = q_1$ on $U_0 \setminus \{\mathbf{0}\}$; since both functions extend smoothly at $\mathbf{0}$, it follows that, up to sign, $h=q_1$ on $U_0$, which completes the proof.  
\end{proof}

Since $J$ is the moment map of a system preserving effective Hamiltonian $S^1$-action near $p$, it follows that its isotropy weights at $p$ equal those of the origin in $\R^4$ with respect to the Hamiltonian $S^1$-action whose moment map is $q_1$ in the above local normal form. In particular, all focus-focus critical points have the same isotropy weights for the $S^1$-action; these are known to be $\{+1,-1\}$ (cf.\ Zung \cite[Th.\ 1.2]{zung_another}).

\begin{remark}\label{rmk:ffw}
  In order to calculate the isotropy weights of an $S^1$-action at a fixed point, an $S^1$-invariant almost complex structure $\mathcal{J}$ has to be fixed. Observe that the (integrable!) almost complex structure used in the proof of Proposition \ref{prop:unique_system_pres_action} is \textit{not} invariant under the $S^1$-action and, as such, cannot be used to compute the weights.
\end{remark}

Below a different proof of the fact that focus-focus critical points have isotropy weights $\{+1,-1\}$ is given; it uses the close relation between a special class of semi-toric systems and symplectic toric manifolds.

\begin{defin}[{\bf Adaptable semi-toric systems}]
\label{def: extendable}
A semi-toric system $(M,\omega,\Phi)$ is called {\em adaptable} if 
the underlying Hamiltonian $S^1$-space $\hamiltonian$ is extendable (cf.\ Definition \ref{ext}). 
\end{defin}

Fix an adaptable system $\semitoric$ and denote the underlying Hamiltonian $S^1$-space by $\hamiltonian$. Theorem \ref{prop: karshon prop 5.21} implies that for all $x \in \R$, $J^{-1}(x)$ contains at most two isolated critical points in $M^{S^1}$. Let $(M,\omega, \mu =(J, \tilde{H}))$ denote a symplectic toric manifold whose Hamiltonian $\T^2$-action extends the one defined by $J$ (this exists by Theorem \ref{prop: karshon prop 5.21}), and let $\Delta$ be the associated Delzant polygon. With this notation in hand, the following proposition can be proved.

\begin{prop}
\label{prop: weights extendable ff}
The isotropy weights of the $S^1$-action at isolated fixed points in $M^{S^1}$ which are of focus-focus type for $\Phi$ are $\{+1,-1\}$.
\end{prop}
\begin{proof}
Any isolated fixed point in $M^{S^1}$ of focus-focus type for $\Phi$ corresponds to a vertex of $\Delta$ (see Figure \ref{fD} below).

\begin{figure}[h] 
  \begin{center} 
    
    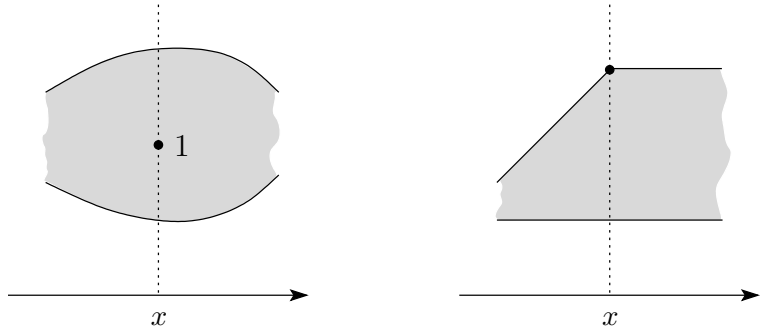
    \caption{The image of a focus-focus point, and the corresponding image in the toric extension.}
    \label{fD} 
    
  \end{center}
\end{figure}

In light of Proposition \ref{prop:unique_system_pres_action} and the subsequent discussion, it suffices to consider an adaptable semi-toric system which has only one focus-focus point, e.g.\ the system considered in Sadovski{\'{\i}} $\&$ Z{\^h}ilinski{\'{\i}} \cite{sad_zhi} or that one constructed in Example \ref{exm:ff}. Fix such a system and let $x \in \R$ be such that $J^{-1}(x)$ contains the only focus-focus point. The Eliasson-Miranda-Zung local normal form implies that $\Jmin<x<\Jmax$, where, as above, $J_{\mathrm{min}}$ (respectively $J_{\mathrm{max}}$) denote the minimum (respectively maximum) value of $J$. Let $\rho$ denote the Duistermaat-Heckman function associated to the $S^1$-action (cf.\ Definition \ref{defin:dhm}). Theorem \ref{th: vu ngoc th 5.3} gives that
\begin{equation}\label{DH}
\rho_J'(x+0)-\rho_J'(x-0)=-1.
\end{equation}
\noindent
Applying Theorem \ref{th: vu ngoc th 5.3} to the symplectic toric manifold $(M,\omega,\mu=(J,\tilde{H}))$, obtain that 
\begin{equation}
  \label{eq:6}
  \rho_J'(x+0)-\rho_J'(x-0)=-e^+(x),
\end{equation}
\noindent
where $e^{+}(x) = \frac{1}{a^+b^+}$, and $a^+,b^+$ are the isotropy weights at $p$ of the $S^1$-action defined by $J$ (note that $p$ is an elliptic-elliptic point for $(M,\omega,\mu=(J,\tilde{H}))$). Observe that the left hand sides of \eqref{DH} and \eqref{eq:6} are equal, as they depend on the $S^1$-action defined by $J$. Therefore, $a^+b^+=-1$, which implies that $\{a^+,b^+\} = \{+1,-1\}$ as required.
\end{proof}

Having found that the isotropy weights at focus-focus critical points are equal to 1 in absolute value, the following corollary is immediate. 

\begin{corollary}\label{cor:zknoff}
  Let $\semitoric$ be a semi-toric system. The poles of a $\Z_k$-sphere of the underlying $S^1$-action are necessarily elliptic-elliptic critical points satisfying condition \ref{item:6}.
\end{corollary}

Corollary \ref{cor:zknoff} begs the question of whether $\polygon$ can be used to calculate the isotropy weights of isolated fixed points in $M^{S^1}$ which are elliptic-elliptic critical points for $\Phi$. Recall that such points are mapped to either Delzant or hidden Delzant vertices of $\polygon$ by $f \circ \Phi$ (cf.\ Remark \ref{rmk:pv}).

\begin{remark}\label{rmk:wtc}
  Given $(M,\omega,\mu = (\mu_1,\mu_2))$ is a symplectic toric manifold, its Delzant polygon $\Delta$ can be used to calculate the isotropy weights of the isolated fixed points of the Hamiltonian $S^1$-action of the associated Hamiltonian $S^1$-space $(M,\omega,\mu_1)$ as follows. By Remark \ref{rk:delham}, such a point maps to a vertex $v_{\Delta}$ of $\Delta$ not incident to a vertical edge. Let $\mathbf{u},\mathbf{w}$ be the primitive integral tangent vectors to the edges incident to $v_{\Delta}$ which
  come out of it. Then the isotropy weights of the chosen $S^1$ action at the corresponding isolated fixed point are given by taking the first coordinates of $\mathbf{u},\mathbf{w}$.
\end{remark}

Let $v$ be a Delzant or hidden Delzant vertex of $\polygon$ satisfying condition \ref{item:6} and choose primitive integral tangent vectors $\mathbf{u}, \mathbf{w}$ to the edges incident to $v$ as in Remark \ref{rmk:wtc}. The next proposition proves that the isotropy weights of the corresponding isolated fixed point in $M^{S^1}$ can be calculated as in the case of symplectic toric manifolds described by Remark \ref{rmk:wtc} above.

\begin{proposition}
\label{prop: elliptic-elliptic weights}
With the notation as above, the isotropy weights of the isolated fixed point corresponding to $v$ are given by the first coordinates of $\mathbf{u}$ and $\mathbf{w}$.
\end{proposition}
\begin{proof}
  If $v$ is a Delzant vertex, then locally $f \circ \Phi$ defines a Hamiltonian $\mathbb{T}^2$-action which extends the $S^1$-action whose moment map is $J$. Thus in this case the result follows, as the action is locally toric and the observations made in Remark \ref{rmk:wtc} hold. On the other hand, if $v$ is hidden Delzant for $\polygon$, then there exists a different choice of cuts $\boldsymbol{\ep}'$ such that $v$ is a Delzant vertex for $\mathcal{P}_{\boldsymbol{\ep}'}$ ($\boldsymbol{\ep}'$ agrees with $\boldsymbol{\ep}$ except that there are no cuts going into $f^{-1}(v)$, cf.\ proof of Lemma \ref{lemma:df}). If $\mathbf{u}, \mathbf{w}$ are the primitive integral tangent vectors to the edges of $\polygon$ chosen as in Remark \ref{rmk:wtc}, then $\mathbf{u}, A_v \mathbf{w}$ are the corresponding ones for $\mathcal{P}_{\boldsymbol{\ep}'}$, where $ A_v = \left(
  \begin{smallmatrix}
  1 & 0 \\ \ep_v n_v & 1
  \end{smallmatrix}
  \right)$,
  $\ep_v, n_v$ being the sign and degree of $v$ (cf.\ proof of Lemma \ref{lemma:df} and \vungoc\ \cite[proof of Prop.\ 4.1]{vu ngoc}). Since the first coordinate of $\mathbf{w}$ agrees with that of $A_v\mathbf{w}$, the result follows from the Delzant case.
\end{proof}


\subsubsection{$\Z_k$-spheres}\label{sec:z_k-spheres} Corollary \ref{cor:zknoff} and Proposition \ref{prop: elliptic-elliptic weights} allow to find poles of $\Z_k$-spheres; what this section is concerned with is to show that, in analogy with the case of symplectic toric manifolds (cf.\ Remark \ref{rk:delham}), these are mapped to the boundary of $\polygon$ under $f \circ \Phi$. Recall that $B = \Phi(M)$ is a curved polygon with curved edges (cf.\ Section \ref{sec:st}).

\begin{proposition}\label{prop:ced}
Let $\Sigma \subset M$ be a $\Z_k$-sphere for $k \geq 2$. Then $\Phi(\Sigma)$ is a curved edge of $B$. 
\end{proposition}

\begin{proof}
  Let $p,q \in \Sigma$ be fixed by the $S^1$-action. By the local normal form of Lemma \ref{lemma: S^1 action normal form}, $p$ and $q$ can be chosen so that one of the isotropy weights of $p$ (respectively $q$) is $k$ (respectively $-k$). Corollary \ref{cor:zknoff} implies that $\Phi(p), \Phi(q)$ are vertices of $B$. Suppose that $ \mathring{B} \cap \Phi(\Sigma)\neq \emptyset$ and consider $s \in \mathring{B} \cap \Phi(\Sigma)$. Suppose that $s$ is a focus-focus critical value, then $\Sigma \cap \Phi^{-1}(s)$ does not contain focus-focus critical points and, therefore, consists only of regular points. However, the $S^1$-action whose moment map is $J$ is free on regular points lying on $\Phi^{-1}(s)$ by the Eliasson-Miranda-Zung local normal form and Proposition \ref{prop:unique_system_pres_action}. This leads to a contradiction, as points on $\Sigma$ are stabilized by $\Z_k \subset S^1$. Therefore, $s$ is not a focus-focus critical value, which implies that $s$ is regular. The Liouville-Arnol'd theorem implies that, locally near $\Phi^{-1}(s)$ there exists a free Hamiltonian $\T^2$-action extending the one defined by $J$, which implies that the stabilizer of points in $\Phi^{-1}(s)$ (with respect to the $S^1$-action) is trivial. Again, this is a contradiction, which implies that $\mathring{B} \cap \Phi(\Sigma) = \emptyset$. Thus $\Phi(\Sigma) \subset \partial B$ and $\Sigma$ consists of elliptic-elliptic and elliptic-regular critical points for $\Phi$. \\

Suppose that $s' \in \Phi(\Sigma)$ is not a vertex of $B$ and let $e$ denote the curved edge containing $s'$. Since focus-focus critical values are isolated, there exists an open neighbourhood $ W \subset B$ of $e$ such that the Hamiltonian action defined by $\Phi$ on $(\Phi^{-1}(W),\omega)$ descends to a Hamiltonian $\T^2$-action, i.e.\ there exists a diffeomorphism $\bar{f}: W \to F(W) \subset \R^2$ onto its image such that $\bar{f} \circ \Phi :(\Phi^{-1}(W),\omega) \to \R^2$ is the moment map of a Hamiltonian $\T^2$-action. Choose $\bar{f}(x,y) = (x,\bar{f}^{(2)}(x,y))$, i.e.\ fix the Hamiltonian vector field of $J|_{\Phi^{-1}(W)}$ to be an infinitesimal generator of the Hamiltonian $\T^2$-action. The Eliasson-Miranda-Zung local normal form implies that $\bar{f}(e)$ is a straight line with integral tangent vector (cf.\ Remark \ref{rmk:bound}). Given the above choices, it follows that a primitive tangent vector $\mathbf{u}$ for $\bar{f}(e)$ is of the form $(k,b)$, for some $b \neq 0$ (cf.\ Remark \ref{rmk:wtc}). It is standard to check that $(\bar{f} \circ \Phi)^{-1}(\bar{f}(e))$ is a $\Z_k$-sphere (cf.\ Karshon \cite{karshon}); since $((\bar{f} \circ \Phi)^{-1}(\bar{f}(e))) \cap \Sigma$ is not empty and contains a point that is not a pole of $\Sigma$, it follows that $\Sigma = (\bar{f} \circ \Phi)^{-1}(\bar{f}(e)) = \Phi^{-1}(e)$, which completes the proof. 
\end{proof}

Unlike the case of symplectic toric manifolds, it is not necessarily true that a $\Z_k$-sphere $\Sigma$ of the underlying Hamiltonian $S^1$-space $\hamiltonian$ of a semi-toric system $\semitoric$ is the preimage of an edge in $\polygon$. This is because some of the cuts may break the curved edge in $B$ whose preimage under $\Phi$ equals $\Sigma$, thereby introducing fake vertices (cf.\ Definition \ref{defn:vertices}). This is illustrated by the following example.

\begin{exm}\label{exm:ff}
Following Pelayo $\&$ \vungoc\ \cite{pelayo_vu_ngoc_con}, the polygons shown in Figure \ref{polytope1} are two semi-toric polygons associated to a semi-toric system $\semitoric$, where $m_f = 1$ and the Taylor series invariant associated to the focus-focus critical point is taken to be 0 (cf.\ Pelayo $\&$ \vungoc\ \cite{pelayo-vu ngoc inventiones,vu_ngoc_ff} for details). The top edge, going from $(0,0)$ to $(2,1)$, of the polygon in Figure \ref{polytope1} (b) corresponds to a $\Z_2$-sphere; however, the same $\Z_2$-sphere is the preimage of the union of the top edges of the polygon in Figure \ref{polytope1} (a).

\begin{figure}[h] 
  \begin{center} 
    
    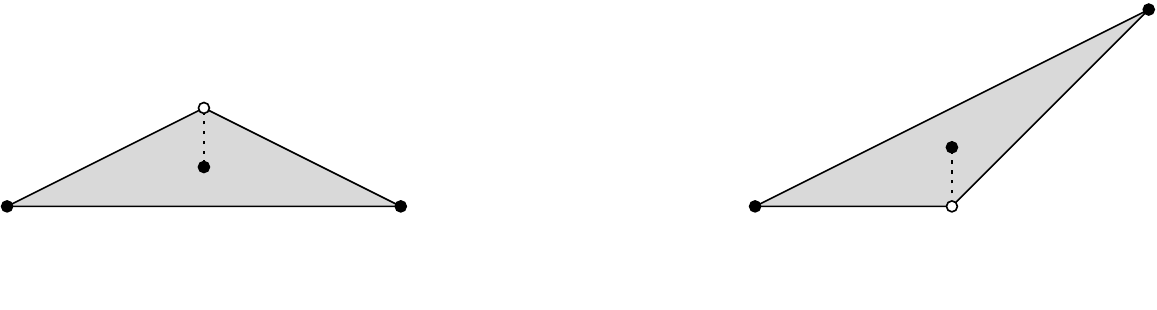
    \caption{}  
    \label{polytope1} 
    
  \end{center}
\end{figure}
\end{exm}

In general, a $\Z_k$-sphere is the preimage of a \textit{chain} of consecutive edges $e_1,\ldots,e_N$ joining two vertices $v, v'$ in $\polygon$, whose `initial' (respectively `final') vertex $v \in e_1$ (respectively $v' \in e_N$) is Delzant or hidden Delzant, has one of its isotropy weights equals to $k$ (respectively $-k$), and whose other vertices are all fake. Note that the isotropy weights at $v$ and $v'$ can be calculated using Proposition \ref{prop: elliptic-elliptic weights}. The adjectives `initial' and `final' refer to direction of increasing first coordinate, which corresponds to the flow of the negative gradient of $J$ with respect to a compatible metric (cf.\ Remark \ref{rmk:grad}).


\subsection{Proof of Theorem \ref{thm: main}}\label{sec:proof-theor-refthm}
Sections \ref{sec:vertex-set-v} and \ref{sec:edge-set-e} allow to describe an algorithm to construct the labeled directed graph $\Gamma$ of the Hamiltonian $S^1$-space $\hamiltonian$ underlying a semi-toric system $\semitoric$ from $\polygon$ and $m_f$. This is explained in the proof of the main theorem below.

\begin{proof}[Proof of Theorem \ref{thm: main}]
  As in Remark \ref{rk:delham}, all that is needed is how to construct the vertices and the edges of $\Gamma$ and their labeling.
\begin{longtable}{l p{90mm}}
{\bf Vertex set V:} & Each Delzant or hidden Delzant vertex $v_{\polygon}$ of $\polygon$ satisfying property \ref{item:6} corresponds to a vertex $v_{\Gamma}$ of $\Gamma$; moreover, there are another $m_f$ vertices of $\Gamma$ each corresponding to one focus-focus critical point of $\Phi$ (cf.\ Lemma \ref{lemma:vertices}). A vertical edge $e^{\mathrm{vert}}_{\polygon}$ of $\polygon$ give rise to a fat vertex $v^{\mathrm{fat}}_{\Gamma}$ of $\Gamma$ (cf.\ Remark \ref{rmk:pfsigma}).\\
{\bf Labeling of V:} & Each vertex $v_{\Gamma}$ (respectively $v^{\mathrm{fat}}_{\Gamma}$) in $V$ is labeled with the value of the first coordinate of the corresponding vertex $v_{\polygon}$ (respectively of the corresponding vertical edge $e^{\mathrm{vert}}_{\polygon}$) in $\polygon$. Fat vertices are also labeled with 0 for the genus of the corresponding fixed surfaces (cf.\ Proposition \ref{prop: fixed surf}) and with the length of the corresponding vertical edge of $\polygon$ for its normalised symplectic area (cf.\ Remark \ref{rmk:pfsigma}). \\
{\bf Edge set E:} & Suppose that $v_{\Gamma} \in V$ has an isotropy weight equal to $k \geq 2$ (note that Corollary \ref{cor:zknoff} implies that $v_{\Gamma}$ does not correspond to a focus-focus critical point). By Proposition \ref{prop: elliptic-elliptic weights} this happens if and only if the corresponding vertex of $v_{\polygon}$ of $\polygon$ has an `outgoing' edge $e_1$ (in the direction of increasing $J$) whose primitive tangent vector is of the form $(k,b)$, for some $b \in \Z$. Construct a chain $C$ of consecutive edges $e_1, e_2, \ldots, e_N$ by moving along $e_1$ in the direction of increasing $J$ until a Delzant or hidden Delzant vertex $v'_{\polygon}$ is reached (this process need terminate). Let $v'_{\Gamma} \in V$ denote the corresponding vertex. Note that by Proposition \ref{prop:ced}, $(f \circ \Phi)^{-1}(C)$ is a $\Z_k$-sphere; thus join $v_{\Gamma}$  to $v_{\Gamma}'$ with an edge $e_{\Gamma}$. \\
{\bf Labeling of E:} & Each edge $e_{\Gamma}$ of $\Gamma$ is labeled with the integer $k \geq 2$, where $(k,b) \in \Z^2$ is a primitive tangent vector to the edge $e_1$ in the corresponding chain $C$ of edges $e_1,\ldots, e_N$ of $\polygon$.
\end{longtable}
\end{proof}

\begin{exm}\label{exm:sssss}
  The semi-toric system whose associated semi-toric polygons are shown in Figure \ref{polytope1} is defined on $(\mathbb{C}\mathrm{P}^2,\omega_{FS})$, where $\omega_{FS}$ is the standard Fubini-Study symplectic form. In fact, the underlying Hamiltonian $S^1$-space is described by Example \ref{exm:cp2}.
\end{exm}


\section{Adaptable and non-adaptable semi-toric systems}\label{sec:adapt-non-adapt}
As remarked above and in the literature, semi-toric systems share many properties with symplectic toric manifolds (cf.\ Remark \ref{rmk:sttoric}, the proof of Theorem \ref{thm: main}, and Pelayo $\&$ \vungoc\ \cite{pelayo-vu ngoc inventiones,pelayo_vu_ngoc_con}, \vungoc\ \cite{vu ngoc}). In light of the classification of symplectic toric manifolds carried out in Delzant \cite{delzant}, it is natural to ask whether semi-toric systems admit a semi-toric polygon which is Delzant in the sense of Definition \ref{defin:delzant_polytopes}. Note that property \ref{item:11} of Theorem \ref{th: vu ngoc th 3.8} implies that a semi-toric polygon $\polygon$ may fail to be Delzant if some vertices are not smooth (cf.\ Definition \ref{defin:delzant_polytopes}).

Recall that $\semitoric$ is adaptable if and only if its underlying Hamiltonian $S^1$-space $\hamiltonian$ is extendable, which in turn means that the $S^1$-action can be extended to an effective Hamiltonian $\mathbb{T}^2$-action on $(M,\omega)$ (cf.\ Definitions \ref{ext} and \ref{def: extendable}, and Theorem \ref{prop: karshon prop 5.21}). 
The aim of this section is to prove the following result.

\begin{theorem}\label{smooth}
A semi-toric system $\semitoric$ admits a Delzant semi-toric polygon $\polygon$ if and only if $\semitoric$ is adaptable.
\end{theorem}

The proof of Theorem \ref{smooth} is obtained by considering the cases of adaptable and non-adaptable semi-toric systems separately (in Sections \ref{sec:adaptable-semi-toric} and \ref{sec:non-adaptable-semi} respectively), and is obtained by combining Corollary \ref{cor:smooth} and Proposition \ref{prop:nonsmoothvertex}. In fact, Corollary \ref{cor:smooth} follows from Theorem \ref{prop:smoothvertex}, which proves a stronger property of adaptable semi-toric systems: Let $\semitoric$ be adaptable, denote by $\hamiltonian$ its underlying Hamiltonian $S^1$-space, and let $(M,\omega, \mu)$ be a symplectic toric manifold whose associated Hamiltonian $S^1$-space (in the sense of Remark \ref{rmk:ext}) is $\hamiltonian$. Then there exists a choice of cuts $\boldsymbol{\ep}$ such that $\polygon = \Delta$, where $\Delta$ is the Delzant polygon classifying $(M,\omega,\mu)$. In other words, the family of semi-toric polygons associated to $\semitoric$ contains all the Delzant polytopes classifying symplectic toric manifolds whose associated Hamiltonian $S^1$-space is $\hamiltonian$. \\

Henceforth, fix a semi-toric system $(M,\omega, \Phi)$. Let $\polygon$ be a semi-toric polygon associated to $\semitoric$. The following lemma gives a necessary and sufficient condition for vertices of $\polygon$ to be smooth.

\begin{lemma}\label{lemma:lff}
  A vertex $v$ of $\polygon$ is smooth if and only if $v$ is either
  \begin{enumerate}[label=(\alph*), ref=(\alph*)]
  \item \label{item:12} Delzant, or
  \item \label{item:13} fake with degree 1 and not lying on a chain of edges corresponding to a $\Z_k$-sphere (cf.\ Definition \ref{defn:vertices} and Section \ref{sec:edge-set-e}).
  \end{enumerate}
\end{lemma} 
\begin{proof}
The idea is to use Lemma \ref{lemma:df} to prove the result. Smoothness of Delzant vertices follows directly from Lemma \ref{lemma:df}; thus it remains to show that hidden Delzant vertices are never smooth and that fake vertices are smooth if and only if they satisfy property \ref{item:13} above. Let $v$ be a vertex of $\polygon$ and let $\mathbf{u},\mathbf{w} \in \Z^2$ be the primitive tangents to the left and right edges of $\polygon$ incident to $v$ with positive first component (the convention is the same as that in the discussion leading to Lemma \ref{lemma:df}). Suppose that $v$ is hidden Delzant. Then Lemma \ref{lemma:df} gives that $\Z\langle \mathbf{u}, A_v \mathbf{w} \rangle = \Z^2$, where 
\begin{equation*}
  A_v=\begin{pmatrix}
    1 & 0 \\
    \ep_v n_v & 1
  \end{pmatrix},
\end{equation*}
\noindent
and $\ep_v$ and $n_v$ are the sign and the degree of $v$ respectively. Let $(\mathbf{u} \,\mathbf{w})$ denote the matrix whose columns are $\mathbf{u}, \mathbf{w}$. Then, if $\mathbf{u} = (u_1,u_2)^T$ and $\mathbf{w} = (w_1,w_2)^T$, a simple calculation shows that
\begin{equation}
  \label{eq:7}
  \det (\mathbf{u}\,A_v\mathbf{w}) = \ep_v n_v u_1 w_1 + \det (\mathbf{u}\,\mathbf{w}).
\end{equation}
\noindent
Since $\ep_v n_v u_1 w_1 \neq 0$, it follows that, if $v$ is smooth, then the signs of $\det (\mathbf{u}\,A_v\mathbf{w})$ and $\det (\mathbf{u}\,\mathbf{w})$ are opposite. Recall that $\mathbf{u}, A_v\mathbf{w}$ are the left (respectively right) primitive tangent vectors to the edges incident to $v$ in a distinct semi-toric polygon $\mathcal{P}_{\boldsymbol{\ep}'}$ whose choice of cuts agrees with $\boldsymbol{\ep}$ except that there are no cuts into $v$ (so that $v$ is a Delzant vertex for $\mathcal{P}_{\boldsymbol{\ep}'}$). Moreover, $\mathcal{P}_{\boldsymbol{\ep}'} = \tau_v(\polygon)$, where $\tau_v$ is a piecewise integral affine transformation which is the identity on the left of vertical line through $v$ and $A_v$ on the right (cf.\ \vungoc\ \cite[Prop. 4.1]{vu ngoc}). This transformation preserves convexity of the semi-toric polygon and, thus, the sign of $\det (\mathbf{u}\,\mathbf{w})$ agrees with that of $\det (\mathbf{u}\,A_v\mathbf{w})$. If $v$ is smooth as a vertex of $\polygon$, this leads to a contradiction. \\

It remains to check that fake vertices are smooth if and only if they satisfy property \ref{item:13}. Suppose $v$ is fake and let $\mathbf{u}, \mathbf{w} \in \Z^2$ be as above. Lemma \ref{lemma:df} gives that $\det  (\mathbf{u}\,A_v\mathbf{w}) = 0$; thus equation \eqref{eq:7} implies that
$$ \det (\mathbf{u}\,\mathbf{w}) = - \ep_v n_v u_1 w_1. $$
\noindent
The vertex $v$ is smooth if and only if $\abs{\det (\mathbf{u}\,\mathbf{w})} =1$, which is equivalent to $n_V, u_1, w_1 = 1$ (since they are all positive integers). The first component of $\mathbf{u},\mathbf{w}$ is equal to 1 if and only if $v$ does not lie on a chain of edges corresponding to a $\Z_k$-sphere (cf.\ proof of Theorem \ref{thm: main}). This completes the proof.
\end{proof}

\subsection{Adaptable semi-toric systems}\label{sec:adaptable-semi-toric}
By condition \ref{item:17} in Theorem \ref{prop: karshon prop 5.21}, $\semitoric$ is adaptable if and only if every non-extremal level set of $J$ contains at most 2 non-free orbits and all fixed surfaces are spheres. By Proposition \ref{prop: fixed surf}, the latter is always satisfied. The only possibilities for non-free orbits of the $S^1$-action which are not extremal are:

\begin{itemize}
\item elliptic-elliptic points not lying on a symplectic sphere fixed by the $S^1$-action,
\item focus-focus points,
\item points lying on a $\Z_k$-sphere but not in $M^{S^1}$ (in light of Proposition \ref{prop:ced}, these lie on elliptic-regular orbits stabilized by a subgroup $\Z_k \subset S^1$).
\end{itemize}

Note that in all the above cases, the non-free orbits of the $S^1$-action consist of critical points of $\Phi$. As before, let $J_{\mathrm{min}}, J_{\mathrm{max}}$ denote the minimum and maximum values of $J$ respectively, and let $x \in \,]J_{\mathrm{min}}, J_{\mathrm{max}}[$. If $J^{-1}(J(x))$ does not contain a focus-focus point, then it contains at most two non-free orbits (cf.\ \vungoc\ \cite[Theorem 3.4]{vu ngoc}). Thus, in order to see whether $\semitoric$ is adaptable or not, it suffices to check that each $J^{-1}(J(x))$ containing at least one focus-focus point does not contain more than two non-free orbits of the $S^1$-action. This happens if and only if $J^{-1}(J(x))$ contains either

\begin{enumerate}[label=(A\arabic*), ref=(A\arabic*)]
\item \label{item:14} exactly one rank 0 critical point of $\Phi$, which is of focus-focus type, and at most one elliptic-regular orbit lying on a $\Z_k$-sphere, or
\item \label{item:16} exactly two rank 0 critical points of $\Phi$, which can be either both of focus-focus type, or one of focus-focus type and the other of elliptic-elliptic type, and no point lying on a $\Z_k$-sphere. 
\end{enumerate}

\begin{remark}\label{rmk:pgad}
  The proof of Theorem \ref{thm: main} implies that any semi-toric polygon $\polygon$ associated to $\semitoric$ can be used to check that the system is adaptable, i.e.\ to check that conditions \ref{item:14} and \ref{item:16} hold.
\end{remark}

Until the end of this section, assume that $\semitoric$ is adaptable with $m_f$ focus-focus critical points, and denote by $\hamiltonian$ its underlying Hamiltonian $S^1$-space. By definition, $\hamiltonian$ is extendable, which, by condition \ref{item:18} in Theorem \ref{prop: karshon prop 5.21}, is equivalent to the existence of an $S^1$-invariant metric with at most two non-trivial chains of gradient spheres (cf.\ Remark \ref{rmk:grad}). In fact, the method employed in Karshon \cite[Prop. 5.16]{karshon} to construct a symplectic toric manifold $(M,\omega, \mu)$ whose associated Hamiltonian $S^1$-space is $\hamiltonian$ is completely determined by 

\begin{enumerate}[label=(K\arabic*), ref=(K\arabic*)]
\item\label{K1} a choice of an $S^1$-invariant metric as above,
\item\label{K2} a suitable toric extension of the $S^1$-action near the minimum of $J$.
\end{enumerate}

Here `suitable' indicates that the moment map associated to the toric extension near the minimum has components $(J,\tilde{H})$ as in Theorem \ref{prop: karshon prop 5.21}. This construction consists of building the two sequences of directed edges $e_1,\ldots,e_r$ and $e'_1,\ldots,e'_{r'}$ starting at the minimum of $J$ and ending at the maximum defining (the boundary of) a Delzant polygon; in particular for each $i>1$,  the tangent to $e_i$ (respectively $e'_i$) is completely determined by the $S^1$-weights of its `initial vertex' (in the direction of $J$ increasing), the direction of $e_{i-1}$ (respectively $e'_{i-1}$) and the convexity of the Delzant polygon. Analogously, a semi-toric polygon associated to $(M,\omega,\Phi)$ is completely determined by

\begin{enumerate}[label=(VN\arabic*), ref=(VN\arabic*)]
\item\label{VN1} a choice of cuts $\boldsymbol{\ep} \in \{+1,-1\}^{m_f}$,
\item\label{VN2} a choice of suitable local action-angle coordinates around the minimum of $J$.
\end{enumerate}
Here `suitable' implies that the angles are defined using the Hamiltonian vector fields $X^J, X^{f^{(2)}(J,H)}$ for some smooth function $f^{(2)}$ as in Theorem \ref{th: vu ngoc th 3.8}. Equivalently, \ref{VN1} and \ref{VN2} determine the homeomorphism $f : \Phi(M) \subset \R^2 \to \R^2$ onto its image $\polygon$ uniquely. \\

It is clear that \ref{K2} and \ref{VN2} are entirely analogous, since a choice of suitable local action-angle coordinates gives a suitable toric extension of the $S^1$-action defined by $J$. The relation between \ref{K1} and \ref{VN1} is explored in the theorem below. Fix choices for \ref{K1} and \ref{K2}, so as to obtain a symplectic toric manifold $(M,\omega,\mu)$ classified by the Delzant polygon $\Delta$.

\begin{theorem}\label{prop:smoothvertex}
  Let $(M, \om, \Phi)$ be adaptable and $(M, \om, J)$ its underlying Hamiltonian $S^1$-space.
  Then there exist choices of cuts $\boldsymbol{\ep} \in \{+1,-1\}$ and of suitable local action-angle coordinates around the minimum of $J$ such that $\polygon = \Delta$.
\end{theorem}

Before proceeding to the proof of theorem \ref{prop:smoothvertex}, note that different choices of metrics satisfying \ref{K1} and different choices of toric extensions near the minimum of $J$, as in \ref{K2}, give rise to all symplectic toric manifolds whose associated Hamiltonian $S^1$-spaces are all isomorphic to $\hamiltonian$. Thus an immediate consequence of Theorem \ref{prop:smoothvertex} is the following corollary.

\begin{corollary}\label{cor:all}
  The family of semi-toric polygons associated to an adaptable system $\semitoric$ contains all Delzant polygons classifying the symplectic toric manifolds whose associated Hamiltonian $S^1$-space is $\hamiltonian$. 
\end{corollary}

\begin{exm}\label{exm:sad_zhi_wow}
  Corollary \ref{cor:all} generalises a phenomenon that occurs for coupled angular momenta on $S^2 \times S^2$: Consider Sadovski{\'{\i}} $\&$ Z{\^h}ilinski{\'{\i}} \cite[Figure 3]{sad_zhi} where the Delzant polygons are the ones on the leftmost and rightmost picture and those obtained by composing those pictures with the transformation $(x,y) \mapsto (x,-y)$. 
\end{exm}

\begin{remark}\label{rmk:nop}
  In order to obtain all Delzant polygons in Corollary \ref{cor:all}, the homeomorphism $f: B \subset \R^2 \to \polygon \subset \R^2$ may have to be chosen to be orientation-reversing (once an orientation in $\R^2$ is fixed), as illustrated by Example \ref{exm:sad_zhi_wow}.
\end{remark}

The main idea behind the proof of Theorem \ref{prop:smoothvertex} is that

\begin{center}
  {\em the chosen metric determines the cuts.}
\end{center}
\noindent
thus illustrating the relation between \ref{K1} and \ref{VN1}. The proof itself proceeds by induction on the number of vertical lines on which the cuts lie. Given the standard coordinates $x,y$ on $\R^2$, let $x_1 \leq x_2 \leq \ldots \leq x_{m_f}$ denote the $x$-coordinates of the focus-focus critical values, and let $J_{\textrm{min}}$ (respectively $J_{\textrm{max}}$) denote the minimum (respectively maximum) value of $J$.  Note that $J_{\mathrm{min}} < x_1$ and that $x_{m_f} < J_{\mathrm{max}}$. Set $N :=|\{x_1,\ldots,x_{m_f}\}| \leq m_f$ and let $\mathsf{x}_1 = x_1 < \mathsf{x}_2 <\ldots <\mathsf{x}_N=x_{m_f}$ denote the distinct values of the $x$-coordinates of the focus-focus critical values. For $i=0,\ldots, N$, define the \emph{slice} $S_i$ of $\Phi(M)$ by
\begin{equation*}\label{slices}
\begin{cases}
  S_0 := \Phi(M)\cap \{(x,y)\in \R^2\mid J_{\textrm{min}} \leq x < \mathsf{x}_{1}\} & \\
  S_i:=\Phi(M)\cap \{(x,y)\in \R^2\mid \mathsf{x}_i < x < \mathsf{x}_{i+1}\} & \mbox{for } i=1,\ldots N-1 \\
  S_{N}:=\Phi(M)\cap \{(x,y)\in \R^2\mid \mathsf{x}_{N} < x \leq J_{\textrm{max}}\} .&			
\end{cases}
\end{equation*}

\begin{proof}[Proof of Theorem \ref{prop:smoothvertex}]
Recall that choices of \ref{K1} and \ref{K2} are fixed, so as to obtain a Delzant polygon $\Delta$ classifying a symplectic toric manifold $(M,\omega,\mu)$. Note that focus-focus critical points for $\semitoric$ map to vertices of $\Delta$ which are not extremal with respect to $J$. Moreover, the $x$ coordinate of the vertices of $\Delta$ is given by the value of $J$ at the corresponding critical point of $\Phi$; in particular, if $c_j = (x_j,y_j)$ is a focus-focus critical value, for $j=1,\ldots, m_f$, then the corresponding vertex of $\Delta$ has coordinates $(x_j, y'_j)$. In analogy with the definition of slices given above, define for $i=0,\ldots,N$ the $\Delta$-slices by
$$
\begin{cases}
  S^{\Delta}_0 := \Delta\cap \{(x,y)\in \R^2\mid J_{\textrm{min}} \leq x <\mathsf{x}_{1}\} & \\
  S^{\Delta}_i:=\Delta\cap \{(x,y)\in \R^2\mid \mathsf{x}_i < x < \mathsf{x}_{i+1}\} & \mbox{for } i=1,\ldots N-1 \\
  S^{\Delta}_{N}:=\Delta\cap \{(x,y)\in \R^2\mid \mathsf{x}_{N} < x \leq J_{\textrm{max}}\} .&			
\end{cases}
$$
\noindent
Note that there is a one-to-one correspondence between slices $S_0, \ldots, S_{N}$ and $\Delta$-slices $S^{\Delta}_0,\ldots, S^{\Delta}_{N}$, which preserves the labeling as illustrated by the figure below.
\begin{figure}[h]
  \begin{center}
    
    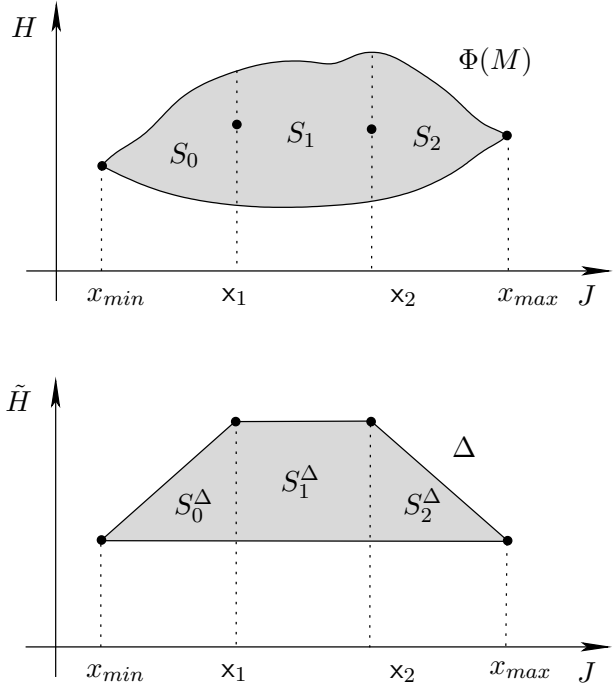
    \caption{Slices and $\Delta$-slices}
    \label{adaptableDelzant}
    
  \end{center}
\end{figure}

The proof proceeds by induction on the number of slices of the curved polygon. If there is only one slice then there is nothing to prove, since the system is of toric type in the sense of \vungoc\  \cite[Def.\  2.1 and Cor. 3.5]{vu ngoc}. Suppose there are slices $S_0, \ldots, S_{N}$ for $N \geq 1$ of the curved polygon. Firstly, observe that the homeomorphism $f_0 := f|_{S_0}$ can be defined in such a way that $f_0(S_0)=S^{\Delta}_0$; this follows from the fact that the action of $\msT$ (cf.\ Remark \ref{rmk:tau}) on $S^{\Delta}_0$ gives all suitable toric extensions of the $S^1$-action defined by $J$ near the minimum and $f_0$ is one such. This is nothing but restating the fact that choices \ref{K2} and \ref{VN2} are equivalent.
Therefore the following may be assumed. \\

\noindent
\textbf{Inductive hypothesis:} A choice of cuts has been made so that $f$ is defined on $\Big(\bigcup\limits_{i=0}^{k-1} \bar{S}_i\Big) \cup S_k$ and that, for all $i=0,\ldots,k$, $f(S_i)= S^{\Delta}_i$. \\

The idea behind the construction of $f$ is to define it slice by slice, at each stage making a choice of cuts so as to obtain a unique extension to the next slice (cf.\ Theorem \ref{th: vu ngoc th 3.8}). Thus the inductive step consists of choosing the cut(s) along the vertical line $\{(x,y) \mid x = \mathsf{x}_{k+1}\}$ in $\Phi(M)$ so that the resulting extension of $f$ on  $\Big(\bigcup\limits_{i=0}^{k} \bar{S}_i\Big) \cup S_{k+1}$ satisfies $f(S_{k+1})= S^{\Delta}_{k+1}$. In order to check this last equality, it suffices to check that the edges incident to the vertices in $f\Big(\Big(\bigcup\limits_{i=0}^{k} \bar{S}_i\Big) \cup S_{k+1}\Big)$ created by the cuts have primitive tangent vectors equal to those of the corresponding vertices in $\Delta$ (up to sign). Note that the choice of cuts is going to be such that there will be a one-to-one correspondence between vertices of $f\Big(\Big(\bigcup\limits_{i=0}^{k} \bar{S}_i\Big) \cup S_{k+1}\Big)$ and of $\Delta$ along the vertical line $\{(x,y) \mid x = \mathsf{x}_{k+1}\}$. Since the semi-toric system is adaptable, condition \ref{item:17} of Theorem \ref{prop: karshon prop 5.21} implies that $J^{-1}(\mathsf{x}_{k+1})$ contains at most two non-free orbits of the $S^1$-action, one of which must be of focus-focus type for $\Phi$ by definition of $\mathsf{x}_{k+1}$. There are two distinct cases to consider, depending on conditions \ref{item:14} and \ref{item:16} described above. \\

\noindent
\underline{Case \ref{item:14}:} exactly one critical point of rank 0 for $\Phi$ in $J^{-1}(\mathsf{x}_{k+1})$. There is only one vertex $v$ of $\Delta$ whose first coordinate equals $\mathsf{x}_{k+1}$; this follows from the description of the isolated fixed points of $M^{S^1}$ (cf.\ Lemma \ref{lemma:vertices} and Theorem \ref{prop: karshon prop 5.21}). Moreover, since this critical point is of focus-focus type (by definition of $\mathsf{x}_{k+1}$), its isotropy weights for the $S^1$-action are $+1, -1$ (cf.\ Proposition \ref{prop: weights extendable ff}). By construction of $\Delta$, the primitive tangent with positive first coordinate $\mathbf{u} \in \Z^2$ (respectively $ \mathbf{w} \in \Z^2$) to the left (respectively right) edge incident to $v$ is of the form $(1, u_2)^T$ (respectively $(1,w_2)^T$). Choose the (only!) cut so that $v$ is a vertex of the image of the extension of $f$ defined on $\Big(\bigcup\limits_{i=0}^{k} \bar{S}_i\Big) \cup S_{k+1}$. This can be achieved as follows. By assumption, there is (part of!) a curved edge in $\Phi(M) \cap S_{k}$ which, under $f$, maps to the edge incident to $v$ on the left (in the direction of increasing $J$). Directing the cut towards this edge is the required choice; extend $f$ to $S_{k+1}$ using the method of Theorem \ref{th: vu ngoc th 3.8}. Observe that, by assumption and by the fact that $f$ is continuous, the vertex of $\Big(\Big(\bigcup\limits_{i=0}^{k} \bar{S}_i\Big) \cup S_{k+1}\Big)$ with first coordinate equal to $\mathsf{x}_{k+1}$ equals $v$, i.e.\ they have the same coordinates. Note that the curved edge mapping to the edge incident to $v$ on the left is not a $\Z_k$-sphere, for, otherwise, any primitive tangent vector of its image would have first component equal to $k \geq 2$ in absolute value (cf.\ the proof of Theorem \ref{thm: main}). However, any such vector equals $ \pm \mathbf{u} = (\pm 1, \pm u_2)^{T}$. In particular, this implies that $v$ is a fake vertex satisfying condition \ref{item:13} of Lemma \ref{lemma:lff} and, thus, it is smooth (for the resulting semi-toric polygon). Since both $\Delta$ and $f\Big(\Big(\bigcup\limits_{i=0}^{k} \bar{S}_i\Big) \cup S_{k+1}\Big)$ are convex, it follows that the edges incident to $v$ on the right in $\Delta$ and $f\Big(\Big(\bigcup\limits_{i=0}^{k} \bar{S}_i\Big) \cup S_{k+1}\Big)$ have equal primitive tangent vectors (up to a sign). \\

\noindent
\underline{Case \ref{item:16}:} exactly two critical points of rank 0 for $\Phi$ in $J^{-1}(\mathsf{x}_{k+1})$. Lemma \ref{lemma:vertices} and Theorem \ref{prop: karshon prop 5.21} imply that $\Delta$ has exactly two vertices $v, v'$ whose first coordinate is $\mathsf{x}_{k+1}$. The corresponding critical points of rank 0 for $\Phi$ are either both of focus-focus type, or one is of focus-focus type and the other is elliptic-elliptic. Accordingly, either $\Phi(M)$ has no vertex with first coordinate equal to $\mathsf{x}_{k+1}$, or exactly one. In the first case, choose the two cuts to go in opposite directions, while in the second, choose the only cut so that it does not go into the vertex of $\Phi(M)$. Extend $f$ as in Theorem \ref{th: vu ngoc th 3.8}. Observe that in either case $f\Big(\Big(\bigcup\limits_{i=0}^{k} \bar{S}_i\Big) \cup S_{k+1}\Big)$ has exactly two vertices with first coordinate equal to $\mathsf{x}_{k+1}$; by assumption and by continuity of $f$, it follows that one of these is equal to $v$, while the other is equal to $v'$, i.e.\ they have the same coordinates. In all cases, the vertices in $f\Big(\Big(\bigcup\limits_{i=0}^{k} \bar{S}_i\Big) \cup S_{k+1}\Big)$ lying on the vertical line $\{(x,y) \mid x = \mathsf{x}_{k+1}\}$ satisfy either property \ref{item:12} (the image of the vertex of $\Phi(M)$ with first coordinate $x = \mathsf{x}_{k+1}$) or \ref{item:13} (all other) of Lemma \ref{lemma:lff} and, thus, they are smooth. Moreover, the left edges incident to these vertices in $f\Big(\Big(\bigcup\limits_{i=0}^{k} \bar{S}_i\Big) \cup S_{k+1}\Big)$ have primitive tangent vectors which agree (up to sign) with those of the edges incident to $v, v'$ in $\Delta$. Again, convexity of $\Delta$ and of $f\Big(\Big(\bigcup\limits_{i=0}^{k} \bar{S}_i\Big) \cup S_{k+1}\Big)$ imply that the same hold for the edges on the right of these vertices.
\end{proof}

The following corollary, which follows at once from Theorem \ref{prop:smoothvertex}, proves the first part of Theorem \ref{smooth}. 

\begin{corollary}\label{cor:smooth}
  An adaptable semi-toric system $(M,\omega,\Phi)$ admits a Delzant semi-toric polygon.
\end{corollary}


\subsection{Non-adaptable semi-toric systems}\label{sec:non-adaptable-semi}
In order to complete the proof of Theorem \ref{smooth}, this section proves that semi-toric systems which fail to be adaptable (henceforth referred to as {\bf non-adaptable}) do not admit any semi-toric polygon whose vertices are all smooth. This is shown via a sequence of simple observations. \\

Before doing so, a useful characterization of non-adaptable semi-toric systems is provided below; this includes both a local and a global condition.
Let $\Sigma$ be an embedded surface in $M$ and denote by $I_\Sigma$ its self-intersection, i.e.\ $I_\Sigma:=\int_\Sigma PD[\Sigma]$, where
$PD[\Sigma]$ is the Poincar\'e dual of $\Sigma$ in $M$.
\purge{
\begin{lemma}\label{lemma Sigma}
Let $(M',\omega)$ be a 4-dimensional symplectic manifold (not necessarily compact) endowed with an effective Hamiltonian $\T^2$-action,
with moment map $\mu=(J,\widetilde{H})\colon M'\to \R^2$. Suppose that $-\infty<J_{\mathrm{min}}$ and $\Sigma=J^{-1}(J_{\mathrm{min}})$ is a ( $\T^2$-invariant ) sphere.
Moreover let  $\mathbf{w},\mathbf{w}'\in \Z^2$ be the primitive integral
 vectors tangent to $e$ and $e'$, the edges coming out of $\mu(\Sigma)$, with positive first coordinate, $\mathbf{v}\in \Z^2$ the primitive integral tangent vector 
 to $\mu(\Sigma)$ with positive second coordinate and $v'$ the vertex at the intersection of the extensions of $e$ and $e'$ (see Fig. \ref{halfpolygon}).
 Then $$| I_\Sigma|=|\det(\mathbf{w}\,\mathbf{w}')| \;.$$ 
 Hence $v'$ is smooth if and only if $|I_{\Sigma}|=1$. 
 
\end{lemma}
\begin{proof}
Let $x,y$ be the standard $\Z$-basis of the lattice $\Z^2$. By the effectiveness of the $\T^2$ action (or equivalently, by the smoothness of the two vertices of $\mu(M')$ corresponding to the images of
the poles of the sphere $\Sigma$)
 we have that $\mathbf{v}=y$, $\mathbf{w}=x+ay$ and $\mathbf{w}'=x+by$ for some $a,b\in \Z$, hence $|\det(\mathbf{w}\,\mathbf{w}')|=|b-a|$.
 
 Let $PD[\Sigma]=[\sigma]\in H^2(M;\Z)$ for some $\sigma\in \Omega^2(M)$ (note that $H^*(M;\Z)$ is torsion free, so we can regard it as a sublattice of $H^2(M;\R)$).
 Since $\Sigma$ is $\T^2$-invariant, we can consider its ``equivariant Poincar\'e dual", namely $PD^{\T^2}[\Sigma]=[\sigma+\mu]\in H^2_{\T^2}(M;\Z)$, where $\sigma+\mu$ is a $2$-form
 in the Cartan complex $(\Omega_{\T^2}(M),d_{\T^2})$ which is $d_{\T^2}$-closed and representing an element in $H_{\T^2}^2(M;\Z)$ (which again can be regarded
 as a sublattice of $H_{\T^2}^2(M;\R)$). Moreover at each fixed point $p$ of the $\T^2$-action on $\Sigma$, $PD^{\T^2}[\Sigma](p)=e_{\T^2}^{N_p}$, the equivariant
 Euler class of the normal bundle $N_p$ of $T_p\Sigma$ in $T_pM$, i.e.\ the product of the
 weights of the $\T^2$ representation on $N_p$. In this case, if we denote by $v_1$ and $v_2$ respectively the south and north pole of $\Sigma$ as in Figure \ref{halfpolygon},
  $PD^{\T^2}[\Sigma](v_1)=\mathbf{w}'$ and  $PD^{\T^2}[\Sigma](v_2)=\mathbf{w}$. Since $\int_{\Sigma}PD[\Sigma]=\int_\Sigma PD^{\T^2}[\Sigma]$, using the
  Atiyah-Bott-Berline-Vergne Localization formula we get
  $$
  I_\Sigma=\int_{\Sigma}PD[\Sigma]=\int_\Sigma PD^{\T^2}[\Sigma]=\frac{\mathbf{w}'}{\mathbf{v}}+\frac{\mathbf{w}}{-\mathbf{v}}=b-a
  $$

\end{proof}

\begin{figure}[h]
\begin{center}
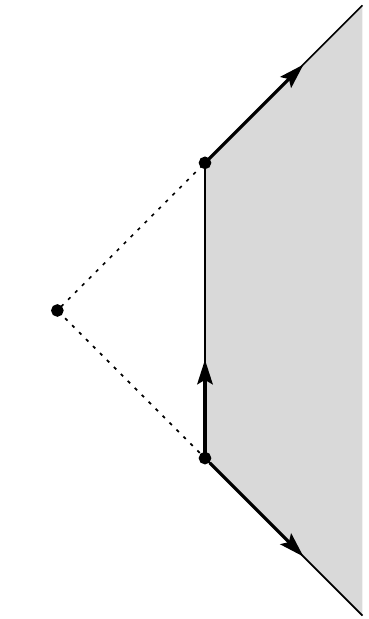
\caption{}
\label{halfpolygon}
\end{center}
\end{figure}

}
\begin{proposition}\label{prop:nad}
  A semi-toric system $\semitoric$ is non-adaptable if and only if
\begin{enumerate}
 \item \emph{Global:} There exists a sphere $\Sigma$ in $M^{S^1}$ (either at the minimum or maximum of $J$) with $ I_\Sigma\neq -1$.
 \item \emph{Local:} There exists $x \in \,]J_{\mathrm{min}},J_{\mathrm{max}}[$ such that $J^{-1}(x)$ contains at least three non-free orbits of the $S^1$-action generated by $J$.
 \end{enumerate}
\end{proposition}
\begin{proof}
\emph{(If)} By Theorem \ref{prop: karshon prop 5.21} \ref{item:17}, condition $(2)$ implies that the system is non-adaptable.

\emph{(Only if)} If the semi-toric system is non-adaptable, then by Theorem \ref{ifp} the $S^1$ action must have fixed surfaces which,
by Proposition \ref{prop: fixed surf}, are symplectic spheres which coincide either with $J^{-1}(J_{\mathrm{min}})$ or $J^{-1}(J_{\mathrm{max}})$. 
Moreover, $(2)$ follows from Theorem \ref{prop: karshon prop 5.21}.
Suppose that each of the $S^1$-fixed spheres has self-intersection $-1$; the idea is to derive a contradiction by constructing a Hamiltonian $S^1$ space 
 with isolated fixed points whose $S^1$ action is not extendable to a $\T^2$ action which, by Theorem \ref{ifp},
is impossible. 

Indeed, let $\Sigma$ be an $S^1$-fixed sphere with $I_{\Sigma}=-1$ and suppose that $\Sigma=J^{-1}(J_{\mathrm{min}})$.
 Let $x_0>J_{\mathrm{min}}$ be such that $\Sigma=M^{S^1}\cap J^{-1}([J_{\mathrm{min}},x_0[)$,
 and let $M_1\subset M$ be the open $S^1$-invariant symplectic submanifold of $M$ given by $ J^{-1}([J_{\mathrm{min}},x_0[)$.
 Since the set of regular values of $\Phi$ in $\Phi(M_1)$ is simply-connected, the $S^1$ action on $M_1$ extends to an effective Hamiltonian $\T^2$ action
 with moment map $\mu_1=(J,H_1)\colon M_1\to \R^2$ (cf.\ \vungoc\ \cite[Prop. 2.12]{vu ngoc}). Thus $\Sigma$ is a $\T^2$-invariant
 symplectic sphere and $\mu_1(\Sigma)$ is a vertical segment. 
 Let $\mathbf{w}=(w_1,w_2)$ and $\mathbf{w}'=(w_1',w_2')$ be the primitive tangent vectors to the edges in $\mu_1(M_1)$ coming out from $\mu_1(\Sigma)$ with positive first coordinate (see Fig. \ref{extension}).
 Endow $\C^2$ with the standard symplectic form, and with a $\T^2$ action given by $(\lambda_1,\lambda_2)\cdot (z_1,z_2)=(\lambda_1^{w_1}\lambda_2^{w_2}z_1,\lambda_1^{w'_1}\lambda_2^{w'_2}z_2)$.
 Since $I_\Sigma=-1$, by the equivariant Darboux-Weinstein theorem
$M_1$ is equivariantly symplectomorphic to a neighborhood of the exceptional divisor 
 in the standard blow up of $\C^2$. As such, $M$ can be equivariantly blown down along $\Sigma$, thus obtaining a new Hamiltonian $S^1$-space
 $(\widetilde{M},\widetilde{\omega},\widetilde{J})$ with an $S^1$-invariant open symplectic submanifold $M_2\subset \widetilde{M}$ where the
 $S^1$ action extends to a Hamiltonian $\T^2$ action with moment map $\mu_2=(\widetilde{J},H_2)\colon M_2\to \R^2$ with exactly one fixed point
 $p$ whose image is $v$
 (see Fig. \ref{extension}), and such that there exists an $S^1$ equivariant symplectomorphism between $M\setminus \bar{M}_1$ and $\widetilde{M}\setminus \bar{M}_2$, where $\bar{M}_i$ denotes the closure of $M_i$ for $i=1,2$.
Note that property (2) still holds for the Hamiltonian $S^1$-space $(\widetilde{M},\widetilde{\omega},\widetilde{J})$.
By repeating the same argument at $J^{-1}(J_{\mathrm{max}})$ if necessary, this procedure yields a Hamiltonian $S^1$-space with isolated fixed points
such that the pre-image of a non-extremal value of the $S^1$ moment map contains at least three non free-orbits
 which, by Theorem \ref{prop: karshon prop 5.21}, implies that the $S^1$ action is non-extendable. However, by Theorem \ref{ifp}, this is
 impossible.
\end{proof}

\begin{figure}[h]
\begin{center}
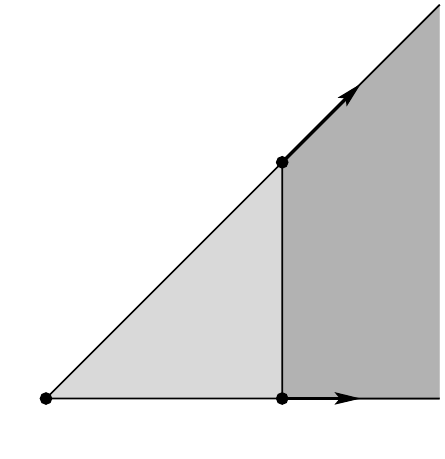
\caption{}
\label{extension}
\end{center}
\end{figure}

\begin{remark}
It can be proved that $|I_\Sigma|=|\det(\mathbf{w}\,\mathbf{w'})|$. Hence the smoothness of $v$ in $\mu_2(M_2)$
 is related to the self-intersection of $\Sigma$.
 
\end{remark}
The following example illustrates a non-adaptable semi-toric system.

\begin{example}\label{example: na}
Consider $\C P^1\times \C P^1$ blown up at two points with a Hamiltonian $\T^2$ action such that the image of the moment map is as in Figure
\ref{ex:nonad} (a). 

By performing a nodal trade as in Symington \cite[Lemma 6.3]{symington}, the polygon in Figure \ref{ex:nonad} (b) gives rise to a semi-toric system. This can also be seen in a different fashion. The polygon in Figure \ref{ex:nonad} (c) is a (representative of a) Delzant semi-toric polygon of complexity 1 as in Pelayo $\&$ \vungoc\ \cite[Def. 4.3]{pelayo_vu_ngoc_con}. Setting the Taylor series invariant associated to the interior marked point to be 0 (cf. Pelayo $\&$ \vungoc\ \cite{pelayo-vu ngoc inventiones}), the resulting decorated polygon gives rise to a semi-toric system using Pelayo $\&$ \vungoc\ \cite[Th. 4.6]{pelayo_vu_ngoc_con}. The claim follows by noticing that the polygons in Figure \ref{ex:nonad} (b) and (c) differ by an appropriate piecewise integral affine transformation (cf. \vungoc\ \cite[Prop. 4.1]{vu ngoc}).
Note that the vertex $(1,3)$ in Figure \ref{ex:nonad} (c) is hidden Delzant.
By blowing up the vertex $(0,0)$ we obtain the generalized polygon in Figure \ref{ex:nonad} (d), corresponding to a semi-toric system
$\Phi=(J,H)\colon M\to \R^2$
with two elliptic-elliptic points and a focus focus point on $J^{-1}(1)$.
Note that the self-intersection of $\Sigma=J^{-1}(J_{\mathrm{min}})$ is $-2$, whereas the one of $\Sigma'=J^{-1}(J_{\mathrm{max}})$ is $-1$; so only
$\Sigma'$ could be blown down.\\
Observe that the labeled graph associated to the underlying Hamiltonian $S^1$-space (with the appropriate choice of symplectic form and moment map) is the same as the one of Example \ref{example:ne}, thus making these two spaces isomorphic in the $\Hs$ category.

Iterating the above procedure, one can prove the existence of a semi-toric system with two elliptic-elliptic points and an arbitrary
number of focus focus points in $J^{-1}(x_0)$, for some $x_0\in \, ]J_{\mathrm{min}},J_{\mathrm{max}}[$.
\end{example}

\begin{figure}[h]
\begin{center}
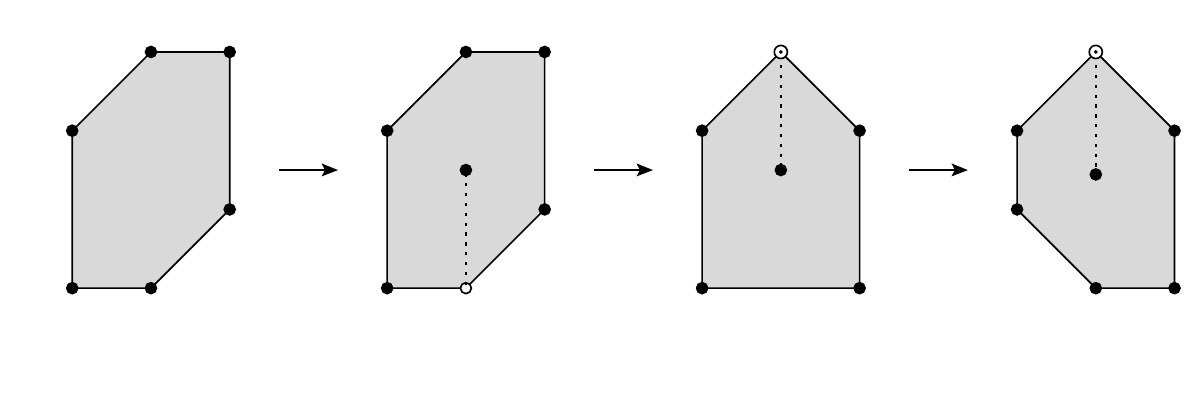
\caption{}
\label{ex:nonad}
\end{center}
\end{figure}

For each $x \in \R$, denote by $E_x, FF_x$ the number of elliptic-elliptic and focus-focus critical points of $\Phi$ in $J^{-1}(x)$. Moreover, let $S_x$ denote the number of elliptic-regular orbits of $\Phi$ stabilized by a subgroup $\Z_k \subset S^1$ in $J^{-1}(x)$, where $k \geq 2$. 

\begin{lemma}\label{sp}
Let $\semitoric$ be a semi-toric system and let $x \in \R$. Then $E_x\leq 2$ and $S_x\leq 2$.
\end{lemma}
\begin{proof}
Lemma \ref{lemma:vertices} implies that if $p \in M^{S^1}$ is an isolated fixed point for the $S^1$-action defined by $J$, then $p$ is either a critical point of focus-focus type or of elliptic-elliptic type for $\Phi$. Since $\Phi(J^{-1}(x))$ contains at most two vertices of $\Phi(M)$ and the fibres of $\Phi$ are connected, it follows that $E_x\leq 2$. The second inequality follows easily by \vungoc\ \cite[Th.\ 3.4]{vu ngoc} and by observing that the image of a $\Z_k$-sphere is, by Proposition \ref{prop:ced}, a curved edge of $\Phi(M)$.
\end{proof}

For non-adaptable semi-toric systems, the following lemma provides some further bounds on the quantities $E_x, FF_x, S_x$ introduced above. 
\begin{lemma}\label{efs}
Let $\semitoric$ be non-adaptable. Let $x\in \R$ be a point such that $E_x+FF_x+S_x\geq 3$. Then 
\begin{itemize}
\item[(i)] If  $E_x+FF_x\geq 3$ then $FF_x\geq 1$.
\item[(ii)] $E_x+S_x\leq 2$.
\end{itemize}
\end{lemma}

\begin{proof}
\emph{(i)}
This follows from Lemma \ref{sp}.\\
\emph{(ii)} By Lemma \ref{sp} it is sufficient to prove that we cannot have $E_x= 2$ and $S_x\geq 1$, or $S_x=2$ and $E_x\geq 1$. This follows easily from the fact that $\Phi(M)$ is a curved polygon, with vertices corresponding to elliptic-elliptic points, and such that the image of $\Z_k$-spheres correspond to edges of $\Phi(M)$ (cf.\ \vungoc\ \cite[Th. 3.4]{vu ngoc} and Proposition \ref{prop:ced}).
\end{proof}

Using Lemma \ref{efs}, the following corollary is immediate.
\begin{corollary}\label{1-6}
For a non-adaptable semi-toric system $\semitoric$, there is a point $x \in \R$ such that exactly one of the following holds.
\begin{enumerate}[label=(\arabic*), ref=(\arabic*)]
\item \label{item:19}
  $E_x=0$, $FF_x\geq 3$ and $S_x=0$.
\item
  $E_x= 1$, $FF_x\geq 2$ and $S_x=0$.
\item
  $E_x=2$, $FF_x\geq 1$ and $S_x=0$.
\item 
  $E_x=0$, $FF_x\geq 1$ and $S_x=2$.
\item 
  $E_x=0$, $FF_x\geq 2$ and $S_x=1$.
\item \label{item:20}
  $E_x=1$, $FF_x\geq 1$ and $S_x=1$.
\end{enumerate}
\end{corollary}

The following proposition completes the proof of Theorem \ref{smooth}.
\begin{prop}
\label{prop:nonsmoothvertex}
Any semi-toric polygon $\polygon$ associated to a non-adaptable $\semitoric$ has at least one vertex that is not smooth.
\end{prop}
\begin{proof}
  It suffices to show that for any choice of cuts $\boldsymbol{\ep}$, one of the vertices of the corresponding semi-toric polygon $\polygon$ fails to be smooth. Let $x \in \R$ be such that $J^{-1}(x)$ contains at least three non-free orbits for the $S^1$-action defined by $J$. Then, by Corollary \ref{1-6}, one of the situations from \ref{item:19} to \ref{item:20} above arises. It is easy to check that in any of the cases \ref{item:19} -- \ref{item:20}, for any choice of cuts there is at least one vertex that does not satisfy either condition \ref{item:12} or \ref{item:13} in Lemma \ref{lemma:lff}. Thus, for any choice of $\boldsymbol{\ep}$, the corresponding semi-toric $\polygon$ has at least one vertex that is not smooth by Lemma \ref{lemma:lff}.
\end{proof}


\end{document}